\documentclass[british,a4paper]{scrartcl}
\usepackage[T1]{fontenc}
\usepackage[latin9]{inputenc}
\usepackage[a4paper]{geometry}
\usepackage{babel}
\usepackage{verbatim}
\usepackage{prettyref}
\usepackage{mathtools}
\usepackage{amsmath}
\usepackage{amsthm}
\usepackage{amssymb}
\usepackage[unicode=true,pdfusetitle,
 bookmarks=true,bookmarksnumbered=false,bookmarksopen=false,
 breaklinks=false,pdfborder={0 0 1},backref=false,colorlinks=false]
 {hyperref}

\makeatletter
\theoremstyle{plain}
\newtheorem{thm}{\protect\theoremname}[section]
\theoremstyle{definition}
\newtheorem{problem}[thm]{\protect\problemname}
\theoremstyle{plain}
\newtheorem{lem}[thm]{\protect\lemmaname}
\theoremstyle{remark}
\newtheorem{rem}[thm]{\protect\remarkname}
\theoremstyle{plain}
\newtheorem{prop}[thm]{\protect\propositionname}
\theoremstyle{plain}
\newtheorem{cor}[thm]{\protect\corollaryname}
\theoremstyle{definition}
\newtheorem{example}[thm]{\protect\examplename}
\theoremstyle{definition}
\newtheorem{defn}[thm]{\protect\definitionname}

\@ifundefined{date}{}{\date{}}
\usepackage{prettyref}

\usepackage{enumerate}
\allowdisplaybreaks

\newcommand*{\e}{\mathrm{e}}
\renewcommand*{\i}{\mathrm{i}}

\newcommand{\R}{\mathbb{R}}
\newcommand{\N}{\mathbb{N}}
\newcommand{\Z}{\mathbb{Z}}
\newcommand{\C}{\mathbb{C}}
\newcommand{\dom}{\operatorname{dom}}
\newcommand{\ran}{\operatorname{ran}}

\renewcommand{\d}{\,\mathrm{d}}
\renewcommand{\Re}{\operatorname{Re}}
\newcommand{\var}{\operatorname{var}}
\renewcommand{\Im}{\operatorname{Im}}

\renewcommand{\tilde}{\widetilde}

\newcommand{\rk}{\operatorname{rk}}
\newcommand{\diag}{\operatorname{diag}}
\newcommand{\D}{\mathrm{D}}
\newcommand{\tr}{\operatorname{tr}}
\newcommand{\Span}{\operatorname{lin}}

\newrefformat{subsec}{Subsection \ref{#1}}
\newrefformat{prob}{Problem \ref{#1}}
\newrefformat{prop}{Proposition \ref{#1}}
\newrefformat{lem}{Lemma \ref{#1}}
\newrefformat{thm}{Theorem \ref{#1}}
\newrefformat{cor}{Corollary \ref{#1}}
\newrefformat{rem}{Remark \ref{#1}}
\newrefformat{exa}{Example \ref{#1}}
\newrefformat{sub}{Subsection \ref{#1}}
\newrefformat{eq}{(\ref{#1})}

\theoremstyle{definition}

\newrefformat{hyp}{Hypotheses \ref{#1}}

\makeatother

\providecommand{\corollaryname}{Corollary}
\providecommand{\definitionname}{Definition}
\providecommand{\examplename}{Example}
\providecommand{\lemmaname}{Lemma}
\providecommand{\problemname}{Problem}
\providecommand{\propositionname}{Proposition}
\providecommand{\remarkname}{Remark}
\providecommand{\theoremname}{Theorem}

\begin{document}
\title{Characterisation for Exponential Stability of port-Hamiltonian Systems}
\author{Sascha Trostorff\thanks{Mathematisches Seminar, CAU Kiel, Germany},
Marcus Waurick\thanks{Institut für Angewandte Analysis, TU Bergakademie Freiberg, Germany}}
\maketitle
\begin{abstract}
Given an energy-dissipating port-Hamiltonian system, we characterise
the exponential decay of the energy via the model ingredients under
mild conditions on the Hamiltonian density $\mathcal{H}$. In passing,
we obtain generalisations for sufficient criteria in the literature
by making regularity requirements for the Hamiltonian density largely
obsolete. The key assumption for the characterisation (and thus the
sufficient criteria) to work is a uniform bound for a family of fundamental
solutions for some non-autonomous, finite-dimensional ODEs. Regularity
conditions on $\mathcal{H}$ for previously known criteria such as
bounded variation are shown to imply the key assumption. Exponentially
stable port-Hamiltonian systems with densities in $L_{\infty}$ only
are also provided.
\end{abstract}
\label{=00005Cnoindent}\textbf{MSC2020 }93D23, 37K40, 47D06, 34G10

\tableofcontents{}

\section{Introduction}

In \cite{AvdS} van der Schaft described the framework of port-Hamiltonian
systems. It has since then triggered manifold research and ideas.
For this we refer to \cite{JZ,vdS_Jeltsema_2014,Augner,SkrepekDiss,Waurick2019}
and the references therein, see also \cite{Jacob_Zwart2018_survey}
for a survey. The basic idea is to describe a physical -- mostly
energy conserving or at least energy dissipating -- phenomenon in
terms of a partial differential equation in the underlying physical
domain together with suitable boundary conditions. These boundary
conditions now are what is thought of as so-called ports. At these
ports one can steer and measure data. Thus, the basic system and particularly
the one-plus-one-dimensional port-Hamiltonian system serves as a prototype
boundary control system. The emphasis is on hyperbolic type partial
differential equations. Quite naturally, it is of interest to understand
those boundary conditions leading to an evolution of the state variable
to have exponentially decaying energy. More precisely, we consider
the operator 
\begin{equation}
\mathcal{\mathcal{A}}\coloneqq P_{1}\partial_{x}\mathcal{H}+P_{0}\mathcal{H}\tag*{{(A1)}}\label{eq:calA}
\end{equation}
as an operator in $H\coloneqq L_{2}(a,b)^{d}$ (as sets) endowed with
the scalar-product $(u,v)\mapsto\langle\mathcal{H}u,v\rangle_{L_{2}}$.
Here, $(a,b)\subseteq\mathbb{R}$ is a bounded interval, $P_{1}=P_{1}^{*}\in\mathbb{R}^{d\times d}$
is an invertible $d\times d$-matrix, $P_{0}=-P_{0}^{*}\in\mathbb{R}^{d\times d}$
and $\partial_{x}$ is the usual weak derivative operator. The mapping
$\mathcal{H}\colon(a,b)\to\mathbb{R}^{d\times d}$ is measurable,
attains values in the non-negative self-adjoint matrices and is strictly
bounded away from $0$ and bounded above. The port-Hamiltonian operator
$\mathcal{A}$ is now accompanied with suitable boundary conditions
encoded in a full rank matrix $W\in\mathbb{R}^{d\times2d}$ satisfying
\begin{equation}
W\left(\begin{array}{cc}
P_{1} & -P_{1}\\
1_{d} & 1_{d}
\end{array}\right)^{-1}\left(\begin{array}{cc}
0 & 1_{d}\\
1_{d} & 0
\end{array}\right)\left(W\left(\begin{array}{cc}
P_{1} & -P_{1}\\
1_{d} & 1_{d}
\end{array}\right)^{-1}\right)^{\ast}\geq0\tag*{{(WB)}}\label{eq:WB}
\end{equation}
 in the way that
\begin{equation}
\dom(\mathcal{A})=\left\{ u\in H\,;\,\mathcal{H}u\in H^{1}(a,b)^{d},\,W\left(\begin{array}{c}
\left(\mathcal{H}u\right)(b)\\
\left(\mathcal{H}u\right)(a)
\end{array}\right)=0\right\} .\tag*{{(A2)}}\label{eq:calAdom}
\end{equation}
The above conditions render $\mathcal{A}$ to be the generator of
a $C_{0}$-semi-group of contractions on $H$ (see e.g. \cite[Theorem 7.2.4]{JZ}
or \cite[Theorem 1.1]{JMZ_2015}, note also \cite[Theorem 1.5]{JMZ_2015}
for conditions for $\mathcal{A}$ being the generator of a general
$C_{0}$-semi-group). For the theory of $C_{0}$-semi-groups we refer
to the monographs \cite{Engel_Nagel2000,Pazy}. It is the aim of this
article to characterise all such settings for which this semi-group
admits exponential decay. The details of the definitions are given
in \prettyref{sec:Preliminaries}. For the time being we comment on
approaches available in the literature and difficulties as well as
elements of our strategy to obtain our characterisation result. We
remark that stability and stabilisation of port-Hamiltonian systems
is an important topic in control theory, see e.g. \cite{Villegas_2009,Ramirez2014,Ramirez_2017,Augner2019,JSchmid}.

By the celebrated Lumer--Phillips theorem (see, e.g., \cite[3.15 Theorem]{Engel_Nagel2000})
for $\mathcal{A}$ to generate a semi-group of contractions it is
equivalent that $\mathcal{A}$ is $m$-dissipative. This property
is independent of the Hamiltonian density $\mathcal{H}$ encoding
the material coefficients in actual physical systems. Hence, well-posedness
and energy dissipation is not hinging on the actual measurements of
the material parameters. Thus, one might think that, too, exponential
stability is independent of the Hamiltonian. This is, however, not
the case as the example in \cite[Section 5]{Engel2013} demonstrates.
This is even more so surprising as the up to now only\footnote{If $\mathcal{H}$ is smooth, a strict inequality in \ref{eq:WB} leads
to exponential stability for the port-Hamiltonian semi-group as well.
However, by \cite[Lemma 9.1.4]{JZ} and its proof, a strict inequality
in \ref{eq:WB} implies the validity of the  condition in \prettyref{thm:suffcrit}.} available condition yielding an exponentially stable semi-group is
(almost) independent of the actual Hamiltonian.
\begin{thm}[{\cite[Theorem 3.5]{JSchmid}}]
\label{thm:suffcrit}Let $\mathcal{H}$ be of bounded variation and
$\mathcal{A}$ given by \ref{eq:calA} and \ref{eq:calAdom} is $m$-dissipative.
If there exist $c>0$ and $\eta\in\{a,b\}$ such that for all $u\in\dom(\mathcal{A})$
\[
\langle u,\mathcal{A}u\rangle_{H}\leq-c\|\mathcal{H}u(\eta)\|^{2},
\]
then $\mathcal{A}$ generates an exponentially stable $C_{0}$-semi-group.
\end{thm}

Note that the example in \cite{Engel2013} assuring the dependence
of $\mathcal{H}$ whether or not the $C_{0}$-semi-group is exponentially
stable, uses constant $\mathcal{H}$. Hence, even in the class of
constant Hamiltonians, the above condition is not a characterisation.
Also, as a second drawback of the results available in the literature,
the Hamiltonian always needs to satisfy certain regularity requirements.
Apart from the more recent advancement in \prettyref{thm:suffcrit},
the results in \cite[Theorem 9.1.3]{JZ} require continuous differentiability
or Lipschitz continuity, see \cite[Theorem III.2]{Villegas_2009}.
We refer to the results in \cite{Augner} for non-autonomous set ups.

The present article aims at replacing the regularity conditions altogether.
For this define $\Phi_{t}$ to be the fundamental solution associated
with 
\[
u'(x)=\i tP_{1}^{-1}(\mathcal{H}(x)^{-1}-P_{0})u(x)\quad(x\in(a,b))
\]
with $\Phi_{t}(a)=1_{d}\coloneqq\diag(1,\ldots,1)\in\R^{d\times d}.$
The key assumption we shall impose here is
\begin{equation}
\sup_{t\in\mathbb{R}}\|\Phi_{t}\|_{\infty}=\sup_{t\in\mathbb{R}}\sup_{x\in(a,b)}\|\Phi_{t}(x)\|<\infty.\tag*{{(B)}}\label{eq:(B)}
\end{equation}
We shall see below that \ref{eq:(B)} is satisfied in many relevant
cases, e.g., if $\mathcal{H}$ is scalar or of bounded variation.
The main theorem of the present article is a characterisation result
of exponential stability in case \ref{eq:(B)} is satisfied.
\begin{thm}
\label{thm:mainresultA} Assume condition \ref{eq:(B)} and that $\mathcal{A}$
given by \ref{eq:calA} and \ref{eq:calAdom} is $m$-dissipative;
i.e. $\mathcal{A}$ generates a contraction semi-group. Then the following
conditions are equivalent:

\begin{enumerate}[(i)]

\item $\mathcal{A}$ generates an exponentially stable $C_{0}$-semi-group.

\item For all $t\in\R$ 
\[
T_{t}\coloneqq W\left(\begin{array}{c}
\Phi_{t}(b)\\
1_{d}
\end{array}\right)
\]
is invertible with $\sup_{t\in\R}\|T_{t}^{-1}\|<\infty.$

\end{enumerate}
\end{thm}

To the best of our knowledge this is the first characterisation result
for exponential stability of port-Hamiltonian systems. Meanwhile,
however, that asymptotic stability has been characterised, though,
see \cite{WZ22}. Together with a set of examples warranting condition
\ref{eq:(B)}, \prettyref{thm:mainresultA} contains all the available
sufficient criteria for exponential stability of port-Hamiltonian
systems as respective special cases as we will demonstrate below.
In particular, we shall provide a condition on the connection of $\mathcal{H}$
and $P_{1}$ guaranteeing the satisfaction of \ref{eq:(B)} -- independently
of any regularity requirements for $\mathcal{H}$. A special case
for this setting is that of \emph{scalar-valued }Hamiltonian densities
$\mathcal{H},$ that is, when we find a bounded scalar function $h\colon(a,b)\to\mathbb{R}$
such that $\mathcal{H}(x)=h(x)1_{d}$ for almost every $x$. The corresponding
theorem characterising exponential stability is then, in fact, fairly
independent of $h$ in the following sense:
\begin{thm}
\label{thm:scalarH}Assume $\mathcal{H}$ to be scalar-valued and
that $\mathcal{A}$ given by \ref{eq:calA} with $P_{0}=0$ and \ref{eq:calAdom}
is $m$-dissipative. Then the following conditions are equivalent

\begin{enumerate}[(i)]

\item $\mathcal{A}$ generates an exponentially stable $C_{0}$-semi-group.

\item For all $t\in\R$ 
\[
\tau_{t}\coloneqq W\left(\begin{array}{c}
\e^{\i tP_{1}^{-1}}\\
1_{d}
\end{array}\right)
\]
is invertible with $\sup_{t\in\R}\|\tau_{t}^{-1}\|<\infty.$

\end{enumerate}
\end{thm}

The sufficient criterion \prettyref{thm:suffcrit} for scalar-valued
$\mathcal{H}$ reads as follows.
\begin{thm}
\label{thm:suffcrit-3}Let $\mathcal{H}$ be scalar-valued and $\mathcal{A}$
given by \ref{eq:calA} and \ref{eq:calAdom} is $m$-dissipative.
If there exist $c>0$ and $\eta\in\{a,b\}$ such that for all $u\in\dom(\mathcal{A})$
\[
\langle u,\mathcal{A}u\rangle_{H}\leq-c\|(\mathcal{H}u)(\eta)\|^{2},
\]
then $\mathcal{A}$ generates an exponentially stable $C_{0}$-semi-group.
\end{thm}

We will demonstrate below that the  condition in \cite[Theorem 3.5]{JSchmid}
implies (ii) from \prettyref{thm:mainresultA}, thus providing an
independent proof of \cite[Theorem 3.5]{JSchmid} in a more general
situation. By means of counterexamples, we will show that the invertibility
of $T_{t}$ is not sufficient for the uniform boundedness of $T_{t}^{-1}.$
It will rely on future research to assess whether condition \ref{eq:(B)}
is needed at all in the present analysis. However, as we will illustrate
in Section 6, it is satisfied under very mild conditions on the operators
involved, which covers most (if not all) systems considered in the
literature so far. Nevertheless, from a mathematical point of view,
it is still interesting to study the necessity of condition \ref{eq:(B)}.
This leads us to two open problems:
\begin{problem}
Characterise all $P_{1}$ and $\mathcal{H}$ such that \ref{eq:(B)}
holds.
\end{problem}

It may well be that \ref{eq:(B)} is not needed altogether:
\begin{problem}
Is \ref{eq:(B)} necessary for exponential stability of the port-Hamiltonian
system at hand?
\end{problem}

We will provide a short outline of the manuscript next. We recall
some of the basic results related to port-Hamiltonian systems relevant
to this article in \prettyref{sec:Preliminaries}. More so, we shall
revisit the generation theorem for port-Hamiltonian systems and rephrase
the same in order to have a better fit to the rationale to follow.
The main result together with its proof is then presented in \prettyref{sec:Characterisation-of-Exponential}.
In \prettyref{sec:FirstApplications}, we specialise the result to
constant energy densities $\mathcal{H}$ and provide some examples
and non-examples. Particularly, we apply our characterisation to \cite[Section 5]{Engel2013}.
\prettyref{sec:A-Sufficient-Criterion} is devoted to frame the already
known positive definiteness type condition at the boundary into the
present setting. More precisely, we shall show that the  criterion
from \prettyref{thm:suffcrit} can be derived from our characterisation
in \prettyref{sec:A-Sufficient-Criterion}. \prettyref{sec:CondB}
is devoted to a discussion of \ref{eq:(B)}. We sum up our findings
in the concluding \prettyref{sec:Conclusion}.

\section{Generation Theorem Revisited\label{sec:Preliminaries}}

In this section, we recall the functional analytic setting of port-Hamiltonian
systems and detail some results from the literature characterising
the generation property of $\mathcal{A}$. Moreover, we slightly reformulate
said generation theorem into a form more suitable for our purposes.
To start out with we fix $d\in\mathbb{N}$ and let $P_{1},P_{0}\in\R^{d\times d}$
be matrices such that $P_{1}=P_{1}^{\ast}$ is invertible and $P_{0}=-P_{0}^{\ast}$.
Moreover, let $\mathcal{H}:(a,b)\to\R^{d\times d}$ be a measurable
function such that
\[
\exists m,M>0\:\forall x\in(a,b):\:m\leq\mathcal{H}(x)=\mathcal{H}(x)^{\ast}\le M,
\]
where the inequalities are meant in the sense of positive definiteness.
We use $\mathcal{H}$ to define a new inner product on $L_{2}(a,b)^{d}$
by setting 
\[
\langle u,v\rangle_{H}\coloneqq\langle\mathcal{H}u,v\rangle_{L_{2}(a,b)^{d}}\quad(u,v\in L_{2}(a,b)^{d})
\]
and denote the Hilbert space $L_{2}(a,b)^{d}$ equipped with this
new inner product by $H$. It follows that $H=L_{2}([a,b];\R^{d})$
is equipped with the norm 
\[
\|u\|_{H}\coloneqq\|\mathcal{H}^{\frac{1}{2}}u\|_{L_{2}(a,b;\R^{d})}.
\]
 Finally, we define the \emph{port-Hamiltonian} (operator) 
\begin{align}
\mathcal{A}:\dom(\mathcal{A})\subseteq H & \to H,\nonumber \\
u & \mapsto P_{1}(\mathcal{H}u)'+P_{0}\mathcal{H}u,\label{eq:pHo}
\end{align}
with a suitable domain $\dom(\mathcal{A})$ satisfying
\[
\dom(\partial_{0}\mathcal{H})\subseteq\dom(A)\subseteq\dom(\partial\mathcal{H}),
\]
where $\partial_{0}$ denotes the distributional derivative on $L_{2}(a,b)^{d}$
with domain $H_{0}^{1}(a,b)^{d}$. We recall the following characterisation
for $m$-dissipativity of port-Hamiltonians:
\begin{thm}[{\cite[Theorem 7.2.4]{JZ}}]
\label{thm:GenContr}Let $\mathcal{A}$ be as in \prettyref{eq:pHo}.
Then $\mathcal{A}$ is $m$-dissipative (and hence, $\mathcal{A}$
generates a contraction semi-group) if and only if there exists a
matrix $W\in\R^{d\times2d}$ with $\rk W=d$ and 
\begin{equation}
W\left(\begin{array}{cc}
P_{1} & -P_{1}\\
1_{d} & 1_{d}
\end{array}\right)^{-1}\left(\begin{array}{cc}
0 & 1_{d}\\
1_{d} & 0
\end{array}\right)\left(W\left(\begin{array}{cc}
P_{1} & -P_{1}\\
1_{d} & 1_{d}
\end{array}\right)^{-1}\right)^{\ast}\geq0\label{eq:posdefWB}
\end{equation}
such that 
\begin{equation}
\dom(\mathcal{A})=\left\{ u\in H\,;\,\mathcal{H}u\in H^{1}(a,b)^{d},\,W\left(\begin{array}{c}
\left(\mathcal{H}u\right)(b)\\
\left(\mathcal{H}u\right)(a)
\end{array}\right)=0\right\} .\label{eq:domain_of_A}
\end{equation}
\end{thm}

The desired reformulation of \prettyref{thm:GenContr} requires some
prerequisites. For this, note since $P_{1}$ is self-adjoint and invertible,
we can decompose the space $\R^{d}$ into $\R^{d}=E_{+}\oplus E_{-}$
with 
\begin{align*}
E_{+} & \coloneqq\Span\{v\in\R^{d}\,;\,\exists\lambda>0:P_{1}v=\lambda v\},\\
E_{-} & \coloneqq\Span\{w\in\R^{d}\,;\,\exists\lambda<0:P_{1}w=\lambda w\}.
\end{align*}
We denote by $\iota_{\pm}\colon E_{\pm}\to\R^{d}$ the canonical embeddings.
Note that $P_{\pm}\coloneqq\iota_{\pm}\iota_{\pm}^{\ast}:\R^{d}\to\R^{d}$
is then the orthogonal projection on $E_{\pm}$ and $\iota_{\pm}^{\ast}\iota_{\pm}\colon E_{\pm}\to E_{\pm}$
is just the identity. Moreover, we set 
\begin{align*}
P_{1}^{+} & \coloneqq\iota_{+}^{\ast}P_{1}\iota_{+}\colon E_{+}\to E_{+}\\
P_{1}^{-} & \coloneqq\iota_{-}^{\ast}(-P_{1})\iota_{-}\colon E_{-}\to E_{-}.
\end{align*}
Note that $P_{1}^{+}$ and $P_{1}^{-}$ are both strictly positive
self-adjoint operators. Moreover, we set 
\begin{align*}
Q_{+} & \coloneqq\iota_{+}\left(P_{1}^{+}\right)^{\frac{1}{2}}\iota_{+}^{\ast}\colon\R^{d}\to\R^{d},\\
Q_{-} & \coloneqq\iota_{-}\left(P_{1}^{-}\right)^{\frac{1}{2}}\iota_{-}^{\ast}\colon\R^{d}\to\R^{d}.
\end{align*}
We equip the spaces $E_{+}$ and $E_{-}$ with the norms 
\begin{align*}
\|x\|_{E_{+}} & \coloneqq\|\iota_{+}\left(P_{1}^{+}\right)^{-\frac{1}{2}}x\|_{\R^{d}},\\
\|y\|_{E_{-}} & \coloneqq\|\iota_{-}\left(P_{1}^{-}\right)^{-\frac{1}{2}}y\|_{\R^{d}}.
\end{align*}
The next lemma is a standard fact from linear algebra.
\begin{lem}
\label{lem:fun_with_marices}Let $W,\tilde{W}\in\R^{d\times2d}$ such
that $\rk W=d$ and $\ker W=\ker\tilde{W}.$ Then there exists $K\in\R^{d\times d}$
invertible with 
\[
W=K\tilde{W}.
\]
\end{lem}

\begin{proof}
Since $\ker W=\ker\tilde{W}$ and $\dim\ker W=d,$ we infer that both
$W$ and $\tilde{W}$ are onto. Hence, the mappings 
\[
W_{1}\colon\ker(W)^{\bot}\to\R^{d}\text{ and }\tilde{W}_{1}\colon\ker(W)^{\bot}\to\R^{d}
\]
 are bijections. We set $K\coloneqq W_{1}\tilde{W}_{1}^{-1}$ and
obtain
\[
Wu=W_{1}P_{\left(\ker W\right)^{\bot}}u=K\tilde{W}_{1}P_{(\ker W)^{\bot}}u=K\tilde{W}u
\]
for each $u\in\R^{d}$. 
\end{proof}
\begin{lem}
\label{lem:W_vs_M}Let $W\in\R^{d\times2d}.$ Then the following statements
are equivalent

\begin{enumerate}[(i)]

\item $W$ has rank $d$ and satisfies 
\begin{equation}
W\left(\begin{array}{cc}
P_{1} & -P_{1}\\
1 & 1
\end{array}\right)^{-1}\left(\begin{array}{cc}
0 & 1\\
1 & 0
\end{array}\right)\left(W\left(\begin{array}{cc}
P_{1} & -P_{1}\\
1 & 1
\end{array}\right)^{-1}\right)^{\ast}\geq0.\label{eq:pos_def}
\end{equation}

\item There exists a matrix $M\in\R^{d\times d}$ with $\|M\|\leq1$
and an invertible matrix $K\in\R^{d\times d}$ such that 
\[
W=K\left(\begin{array}{cc}
Q_{+}-MQ_{-} & Q_{-}-MQ_{+}\end{array}\right).
\]

\end{enumerate}

Moreover, $\|M\|<1$ is equivalent to strict positive definiteness
in \prettyref{eq:pos_def}.
\end{lem}

\begin{proof}
(i) $\Rightarrow$ (ii): We write $W=\left(\begin{array}{cc}
W_{1} & W_{2}\end{array}\right)$ with $W_{1},W_{2}\in\R^{d\times d}.$ An easy calculation reveals
that \prettyref{eq:pos_def} is equivalent to $W_{2}P_{1}^{-1}W_{2}^{\ast}\leq W_{1}P_{1}^{-1}W_{1}^{\ast}$.
We consider now the adjoint mapping $W^{\ast}\colon\R^{d}\to\R^{2d}.$
For $z=(x,y)\in\ran W^{\ast}$, we find $u\in\R^{d}$ such that $x=W_{1}^{\ast}u$
and $y=W_{2}^{\ast}u.$ The latter gives 
\begin{align*}
\|\iota_{+}^{\text{\ensuremath{\ast}}}y\|_{E_{+}}^{2}-\|\iota_{-}^{\ast}y\|_{E_{-}}^{2} & =\langle\iota_{+}\left(P_{1}^{+}\right)^{-1}\iota_{+}^{\ast}y,y\rangle_{\R^{d}}-\langle\iota_{-}\left(P_{1}^{-}\right)^{-1}\iota_{-}^{\ast}y,y\rangle_{\R^{d}}\\
 & =\langle P_{1}^{-1}y,y\rangle_{\R^{d}}\\
 & =\langle P_{1}^{-1}W_{2}^{\ast}u,W_{2}^{\ast}u\rangle_{\R^{d}}\\
 & =\langle W_{2}P_{1}^{-1}W_{2}^{\ast}u,u\rangle_{\R^{d}}\\
 & \leq\langle W_{1}P_{1}^{-1}W_{1}^{\ast}u,u\rangle_{\R^{d}}\\
 & =\langle P_{1}^{-1}x,x\rangle_{\R^{d}}\\
 & =\|\iota_{+}^{\text{\ensuremath{\ast}}}x\|_{E_{+}}^{2}-\|\iota_{-}^{\ast}x\|_{E_{-}}^{2}
\end{align*}
and thus, 
\begin{equation}
\|\iota_{+}^{\text{\ensuremath{\ast}}}y\|_{E_{+}}^{2}+\|\iota_{-}^{\ast}x\|_{E_{-}}^{2}\leq\|\iota_{+}^{\text{\ensuremath{\ast}}}x\|_{E_{+}}^{2}+\|\iota_{-}^{\ast}y\|_{E_{-}}^{2}.\label{eq:estimate_W^ast-1}
\end{equation}

Next, consider the mapping 
\[
S\coloneqq\left(\begin{array}{cc}
P_{+} & P_{-}\end{array}\right)W^{\ast}\colon\R^{d}\to\R^{d},\quad u\mapsto P_{+}W_{1}^{\ast}u+P_{-}W_{2}^{\ast}u.
\]
This mapping is linear and one-to-one and hence, a bijection. Indeed,
if $Su=P_{+}W_{1}^{\ast}u+P_{-}W_{2}^{\ast}u=0,$ then $P_{+}W_{1}^{\ast}u=0=P_{-}W_{2}^{\ast}u.$
By \prettyref{eq:estimate_W^ast-1} it follows that $P_{+}W_{2}^{\ast}u=0=P_{-}W_{1}^{\ast}u$
and consequently, $W^{\ast}u=0.$ Since $W^{\ast}$ has rank $d$,
it is one-to-one and thus, $u=0.$ We now define the mapping 
\[
C\colon\R^{d}\to\R^{d},\quad v\mapsto\left(\begin{array}{cc}
P_{-} & P_{+}\end{array}\right)W^{\ast}S^{-1}v
\]
and set 
\[
M\coloneqq-\left(Q_{+}+Q_{-}\right)C^{\ast}\left(\iota_{+}\left(P_{1}^{+}\right)^{-\frac{1}{2}}\iota_{+}^{\ast}+\iota_{-}\left(P_{1}^{-}\right)^{-\frac{1}{2}}\iota_{-}^{\ast}\right).
\]
We now prove that the matrices $W$ and $\tilde{W}\coloneqq\left(\begin{array}{cc}
Q_{+}-MQ_{-} & Q_{-}-MQ_{+}\end{array}\right)$ have the same kernels. Indeed, 
\begin{align*}
(x,y)\in\ker W & \Leftrightarrow(x,y)\in\left(\ran W^{\ast}\right)^{\bot},\\
 & \Leftrightarrow\forall u\in\R^{d}:\langle x,W_{1}^{\ast}u\rangle+\langle y,W_{2}^{\ast}u\rangle=0,\\
 & \Leftrightarrow\forall u\in\R^{d}:\langle P_{-}x,P_{-}W_{1}^{\ast}u\rangle+\langle P_{-}y,P_{-}W_{2}^{\ast}u\rangle+\langle P_{+}x,P_{+}W_{1}^{\ast}u\rangle+\langle P_{+}y,P_{+}W_{2}^{\ast}u\rangle=0,\\
 & \Leftrightarrow\forall u\in\R^{d}:\langle P_{-}x+P_{+}y,P_{-}W_{1}^{\ast}u+P_{+}W_{2}^{\ast}u\rangle+\langle P_{+}x+P_{-}y,P_{+}W_{1}^{\ast}u+P_{-}W_{2}^{\ast}u\rangle=0,\\
 & \Leftrightarrow\forall u\in\R^{d}:\langle P_{-}x+P_{+}y,CSu\rangle+\langle P_{+}x+P_{-}y,Su\rangle=0,\\
 & \Leftrightarrow\forall u\in\R^{d}:\langle P_{+}x+P_{-}y+C^{\ast}(P_{-}x+P_{+}y),Su\rangle=0,\\
 & \Leftrightarrow-C^{\ast}(P_{-}x+P_{+}y)=P_{+}x+P_{-}y\\
 & \Leftrightarrow-(Q_{+}+Q_{-})C^{\ast}(P_{-}x+P_{+}y)=Q_{+}x+Q_{-}y\\
 & \Leftrightarrow M\left(Q_{+}y+Q_{-}x\right)=Q_{+}x+Q_{-}y,
\end{align*}
implying
\[
(x,y)\in\ker W\Leftrightarrow\left(Q_{+}-MQ_{-}\right)x+\left(Q_{-}-MQ_{+}\right)y=0;
\]
i.e. $\ker W=\ker\tilde{W}.$ Employing \prettyref{lem:fun_with_marices}
we find $K\in\R^{d\times d}$ invertible such that 
\[
W=K\tilde{W}=K\left(\begin{array}{cc}
Q_{+}-MQ_{-} & Q_{-}-MQ_{+}\end{array}\right).
\]
It remains to show that $\|M\|\leq1.$ For this, let $v\in\R^{d}$.
Since $S$ is onto, we find $u\in\R^{d}$ such that 
\[
Q_{+}v+Q_{-}v=Su=P_{+}W_{1}^{\ast}u+P_{-}W_{2}^{\ast}u=P_{+}x+P_{-}y,
\]
where we set $(x,y)\coloneqq W^{\ast}u\in\ran W^{\ast}$; thus, $Q_{-}v=P_{-}y$
and $Q_{+}v=P_{+}x.$ We compute, using \prettyref{eq:estimate_W^ast-1}
for $(x,y)$
\begin{align*}
\|M^{\ast}v\|^{2} & =\|\left(\iota_{+}\left(P_{1}^{+}\right)^{-\frac{1}{2}}\iota_{+}^{\ast}+\iota_{-}\left(P_{1}^{-}\right)^{-\frac{1}{2}}\iota_{-}^{\ast}\right)C(Q_{+}v+Q_{-}v)\|^{2}\\
 & =\|\left(\iota_{+}\left(P_{1}^{+}\right)^{-\frac{1}{2}}\iota_{+}^{\ast}+\iota_{-}\left(P_{1}^{-}\right)^{-\frac{1}{2}}\iota_{-}^{\ast}\right)CSu\|^{2}\\
 & =\|\left(\iota_{+}\left(P_{1}^{+}\right)^{-\frac{1}{2}}\iota_{+}^{\ast}+\iota_{-}\left(P_{1}^{-}\right)^{-\frac{1}{2}}\iota_{-}^{\ast}\right)(P_{-}x+P_{+}y)\|^{2}\\
 & =\|\iota_{+}^{\ast}y\|_{E_{+}}^{2}+\|\iota_{-}^{\ast}x\|_{E_{-}}^{2}\\
 & \leq\|\iota_{+}^{\ast}x\|_{E_{+}}^{2}+\|\iota_{-}^{\ast}y\|_{E_{-}}^{2}\\
 & =\|\iota_{+}(P_{1}^{+})^{-\frac{1}{2}}\iota_{+}^{\ast}x+\iota_{-}(P_{1}^{-})^{-\frac{1}{2}}\iota_{-}^{\ast}y\|^{2}\\
 & =\|\iota_{+}(P_{1}^{+})^{-\frac{1}{2}}\iota_{+}^{\ast}Q_{+}v+\iota_{-}(P_{1}^{-})^{-\frac{1}{2}}\iota_{-}^{\ast}Q_{-}v\|^{2}\\
 & =\|P_{+}v+P_{-}v\|^{2}=\|v\|^{2},
\end{align*}
which yields $\|M\|=\|M^{\ast}\|\leq1.$\\
(ii) $\Rightarrow$ (i): Assume 
\[
W=K\left(\begin{array}{cc}
Q_{+}-MQ_{-} & Q_{-}-MQ_{+}\end{array}\right)
\]
 for $K,M\in\R^{d\times d}$ with $\|M\|\leq1$ and $K$ invertible.
We show that $W$ has rank $d$. Since $K$ is invertible, it suffices
to show that $\tilde{W}\coloneqq\left(\begin{array}{cc}
Q_{+}-MQ_{-} & Q_{-}-MQ_{+}\end{array}\right)$ has rank $d$, which in turn is equivalent to $\ker\left(\tilde{W}\right)^{\ast}=\{0\}.$
So, let $u\in\ker(\tilde{W})^{\ast}$; that is, 
\[
Q_{+}u=Q_{-}M^{\ast}u,\quad Q_{-}u=Q_{+}M^{\ast}u.
\]
Since $Q_{-}$ and $Q_{+}$ attain values in $E_{-}$ and $E_{+}$,
respectively, we infer $Q_{-}u=Q_{+}u=0$ and hence $P_{-}u=P_{+}u=0,$
which imply $u=0.$ It remains to show \prettyref{eq:pos_def}. As
shown above, this is equivalent to \prettyref{eq:estimate_W^ast-1}.
So let $(x,y)=W^{\ast}u$ for some $u\in\R^{d}$; that is 
\[
x=\left(Q_{+}-Q_{-}M^{\ast}\right)K^{\ast}u,\quad y=\left(Q_{-}-Q_{+}M^{\ast}\right)K^{\ast}u.
\]
Then we compute
\begin{align*}
\|\iota_{+}^{\text{\ensuremath{\ast}}}y\|_{E_{+}}^{2}+\|\iota_{-}^{\ast}x\|_{E_{-}}^{2} & =\|\iota_{+}\left(P_{1}^{+}\right)^{-\frac{1}{2}}\iota_{+}^{\ast}y\|_{\R^{d}}^{2}+\|\iota_{-}\left(P_{1}^{-}\right)^{-\frac{1}{2}}\iota_{-}^{\ast}x\|_{\R^{d}}^{2}\\
 & =\|P_{+}M^{\ast}K^{\ast}u\|_{\R^{d}}^{2}+\|P_{-}M^{\ast}K^{\ast}u\|_{\R^{d}}^{2}\\
 & =\|M^{\ast}K^{\ast}u\|_{\R^{d}}^{2}\\
 & \leq\|K^{\ast}u\|_{\R^{d}}^{2}\\
 & =\|P_{+}K^{\ast}u\|_{\R^{d}}^{2}+\|P_{-}K^{\ast}u\|_{\R^{d}}^{2}\\
 & =\|\iota_{+}\left(P_{1}^{+}\right)^{-\frac{1}{2}}\iota_{+}^{\ast}x\|_{\R^{d}}^{2}+\|\iota_{-}\left(P_{1}^{-}\right)^{-\frac{1}{2}}\iota_{-}^{\ast}y\|_{\R^{d}}^{2}\\
 & =\|\iota_{+}^{\ast}x\|_{E_{+}}^{2}+\|\iota_{-}^{\ast}y\|_{E_{-}}^{2}.
\end{align*}
For the final claim, we note that the strict positive definiteness
in \prettyref{eq:pos_def} is equivalent to 
\[
\|\iota_{+}^{\text{\ensuremath{\ast}}}y\|_{E_{+}}^{2}+\|\iota_{-}^{\ast}x\|_{E_{-}}^{2}<\|\iota_{+}^{\text{\ensuremath{\ast}}}x\|_{E_{+}}^{2}+\|\iota_{-}^{\ast}y\|_{E_{-}}^{2}\quad((x,y)\in\ran W^{\ast}\setminus\{0\}),
\]
which by the computations above is equivalent to $\|M\|=\|M^{\ast}\|<1.$ 
\end{proof}
Using the latter lemma, we obtain the following characterisation result
for $\mathcal{A}$ generating a contraction semi-group.
\begin{thm}
\label{thm:m-accretive-pH}The following statements are equivalent:

\begin{enumerate}[(i)]

\item $\mathcal{A}$ is m-dissipative,

\item $\mathcal{A}$ is dissipative and there exists $W\in\R^{d\times2d}$
such that 
\[
\dom(\mathcal{A})=\left\{ u\in H\,;\,\mathcal{H}u\in H^{1}([a,b];\R^{d}),\,W\left(\begin{array}{c}
(\mathcal{H}u)(b)\\
(\mathcal{H}u)(a)
\end{array}\right)=0\right\} ,
\]

\item there exists $W\in\R^{d\times2d}$ with maximal rank, such
that 
\begin{equation}
W\left(\begin{array}{cc}
P_{1} & -P_{1}\\
1 & 1
\end{array}\right)^{-1}\left(\begin{array}{cc}
0 & 1\\
1 & 0
\end{array}\right)\left(W\left(\begin{array}{cc}
P_{1} & -P_{1}\\
1 & 1
\end{array}\right)^{-1}\right)^{\ast}\geq0\label{eq:positive_W}
\end{equation}
and 
\[
\dom(\mathcal{A})=\left\{ u\in H\,;\,\mathcal{H}u\in H^{1}([a,b];\R^{d}),\,W\left(\begin{array}{c}
(\mathcal{H}u)(b)\\
(\mathcal{H}u)(a)
\end{array}\right)=0\right\} ,
\]

\item There exists a matrix $M\in\R^{d\times d}$ with $\|M\|\leq1$
such that 
\[
\dom(\mathcal{A})=\left\{ u\in H\,;\,\mathcal{H}u\in H^{1}([a,b];\R^{d}),\:Q_{-}(\mathcal{H}u)(a)+Q_{+}(\mathcal{H}u)(b)=M\left(Q_{-}(\mathcal{H}u)(b)+Q_{+}(\mathcal{H}u)(a)\right)\right\} .
\]

\end{enumerate}Moreover, in case (iv) we have 
\[
\langle\mathcal{A}u,u\rangle_{H}=\frac{1}{2}\left(\|M\left(Q_{-}(\mathcal{H}u)(b)+Q_{+}(\mathcal{H}u)(a)\right)\|^{2}-\|Q_{-}\left(\mathcal{H}u\right)(b)+Q_{+}\left(\mathcal{H}u\right)(a)\|^{2}\right)
\]
for each $u\in\dom(\mathcal{A}).$ Furthermore, the matrix $W$ in
(ii) or (iii) can be expressed in terms of the matrix $M$ in (iv)
by 
\[
W=K\left(\begin{array}{cc}
Q_{+}-MQ_{-} & Q_{-}-MQ_{+}\end{array}\right)
\]
for some invertible matrix $K\in\R^{d\times d}.$
\end{thm}

\begin{proof}
The equivalence of statements (i)-(iii) is well-known, cf. \cite[Theorem 1]{JMZ_2015}
and \cite[Theorem 7.2.4]{JZ}. Moreover, the equivalence of (iii)
and (iv) and the relation of the matrices $W$ and $M$ follows from
\prettyref{lem:W_vs_M}. The only thing to be shown is the formula
for $\langle\mathcal{A}u,u\rangle_{H}.$ So, let $u\in\dom(\mathcal{A})$.
Then
\begin{align*}
2\langle\mathcal{A}u,u\rangle_{H} & =2\langle P_{1}\partial_{1}\mathcal{H}u,\mathcal{H}u\rangle+2\langle P_{0}\mathcal{H}u,\mathcal{H}u\rangle\\
 & =2\langle P_{1}\partial_{1}\mathcal{H}u,\mathcal{H}u\rangle\\
 & =\langle P_{1}\left(\mathcal{H}u\right)(b),\left(\mathcal{H}u\right)(b)\rangle-\langle P_{1}\left(\mathcal{H}u\right)(a),\left(\mathcal{H}u\right)(a)\rangle\\
 & =\langle Q_{+}^{2}(\mathcal{H}u)(b),\left(\mathcal{H}u\right)(b)\rangle-\langle Q_{-}^{2}(\mathcal{H}u)(b),\left(\mathcal{H}u\right)(b)\rangle\\
 & \quad-\langle Q_{+}^{2}\left(\mathcal{H}u\right)(a),\left(\mathcal{H}u\right)(a)\rangle+\langle Q_{-}^{2}\left(\mathcal{H}u\right)(a),\left(\mathcal{H}u\right)(a)\rangle\\
 & =\|Q_{+}\left(\mathcal{H}u\right)(b)+Q_{-}\left(\mathcal{H}u\right)(a)\|^{2}-\|Q_{+}\left(\mathcal{H}u\right)(a)+Q_{-}\left(\mathcal{H}u\right)(b)\|^{2}\\
 & =\|M\left(Q_{+}\left(\mathcal{H}u\right)(a)+Q_{-}\left(\mathcal{H}u\right)(b)\right)\|^{2}-\|Q_{+}\left(\mathcal{H}u\right)(a)+Q_{-}\left(\mathcal{H}u\right)(b)\|^{2}
\end{align*}
which shows the assertion.
\end{proof}
\begin{rem}
The anonymous referee kindly provided us with an alternative proof
for the equivalence of (i) and (iv) in \prettyref{thm:m-accretive-pH},
using theory of port-Hamiltonian systems instead of \prettyref{lem:W_vs_M}.
We sketch this proof as follows: By \prettyref{thm:GenContr} item
(i) in \prettyref{thm:m-accretive-pH} is equivalent to $\mathcal{A}$
with domain as in \prettyref{eq:domain_of_A} being a generator of
a contraction semigroup. By unitary equivalence one can assume that
$\mathcal{H}=1_{d}$ (see \cite[Lemma 7.2.3]{JZ} or \prettyref{lem:basictransform}
below). Again by unitary equivalence it suffices to treat the case
$P_{1}=\left(\begin{array}{cc}
\Lambda & 0\\
0 & \Theta
\end{array}\right)=\left(\begin{array}{cc}
P_{1}^{+} & 0\\
0 & -P_{1}^{-}
\end{array}\right)$. In this setup one can apply \cite[Theorem 13.3.1]{JZ}. Indeed,
$\tilde{\mathcal{A}}u=P_{1}u'$ generates a $C_{0}$-semigroup on
$L_{2}(a,b)^{d}$ if and only if its domain is given by 
\[
\dom(\tilde{\mathcal{A}})=\left\{ u\in H^{1}(a,b)^{d}\,;\,KP_{1}\left(\begin{array}{c}
u^{+}(b)\\
u^{-}(a)
\end{array}\right)+LP_{1}\left(\begin{array}{c}
u^{+}(a)\\
u^{-}(b)
\end{array}\right)=0\right\} ,
\]
for some matrices $K,L,$ where $K$ is invertible. Hence, (i) is
equivalent to $\tilde{\mathcal{A}}$ with domain as above being dissipative
(note that an operator is m-dissipative if and only if it is dissipative
and generates a $C_{0}$-semigroup). By invertibility of $K$ the
boundary condition for $\tilde{\mathcal{A}}$ can be rewritten as
\[
(Q_{+}+Q_{-})\left(\begin{array}{c}
u^{+}(b)\\
u^{-}(a)
\end{array}\right)=\sqrt{|P_{1}|}\left(\begin{array}{c}
u^{+}(b)\\
u^{-}(a)
\end{array}\right)=M\sqrt{|P_{1}|}\left(\begin{array}{c}
u^{+}(a)\\
u^{-}(b)
\end{array}\right)=M(Q_{+}+Q_{-})\left(\begin{array}{c}
u^{+}(a)\\
u^{-}(b)
\end{array}\right),
\]
for some matrix $M$. Then (standard) integration by parts reveals
that $\tilde{\mathcal{A}}$ is dissipative if and only if $\|M\|\leq1$
(for the only if part use \prettyref{lem:trace_onto}). 
\end{rem}

\section{Characterisation of Exponential Stability\label{sec:Characterisation-of-Exponential} }

In this section we prove the main theorem of the present manuscript.
We start off with an elementary observation, see also \cite{PTWW_22}.
\begin{lem}
\label{lem:basictransform}Let $a,b\in\R$ with $a<b$ and $d\in\N$.
Moreover, let $\mathcal{H}\in L_{\infty}(a,b;\R^{d\times d})$ with
$0<m\leq\mathcal{H}(x)=\mathcal{H}(x)^{\top}$ for almost every $x\in(a,b)$
and let $H$ be the Hilbert space $L_{2}(a,b;\R^{d})$ endowed with
the inner product $\langle u,v\rangle_{H}\coloneqq\langle\mathcal{H}u,v\rangle_{L_{2}}.$ 

Define 
\[
S:L_{2}(a,b)^{d}\to H,\quad u\mapsto\mathcal{H}^{-1}u.
\]
Then $S$ is a Banach space isomorphism with $S^{\ast}v=v$ for each
$v\in L_{2}(a,b)^{d}.$
\end{lem}

\begin{proof}
$S$ clearly is bounded, one-to-one and onto; i.e., a Banach space
isomorphism. Moreover, for $u\in L_{2}(a,b)^{d},v\in H$ we compute
\[
\langle Su,v\rangle_{H}=\langle\mathcal{H}\mathcal{H}^{-1}u,v\rangle_{L_{2}(a,b)^{d}}=\langle u,v\rangle_{L_{2}(a,b)^{d}},
\]
which shows $S^{\ast}v=v.$
\end{proof}
For $t\in\R$ we consider the following ordinary differential equation
\[
v'(x)=P_{1}^{-1}(\i t\mathcal{H}(x)^{-1}-P_{0})v(x)
\]
and denote by $\Phi_{t}\colon[a,b]\to\R^{d\times d}$ the fundamental
matrix satisfying $\Phi_{t}(a)=1_{d}.$ 
\begin{lem}
\label{lem:inverse_fundamental_matrix}For $t\in\R$ and $x\in[a,b]$
we have $\Phi_{t}(x)^{-1}=P_{1}^{-1}\Phi_{t}(x)^{\ast}P_{1}.$ In
particular $\|\Phi_{t}(x)^{-1}\|\leq\|P_{1}^{-1}\|\|P_{1}\|\|\Phi_{t}(x)\|.$ 
\end{lem}

\begin{proof}
We compute the derivative of $\Psi\colon x\mapsto\Phi_{t}(x)^{-1}$.
We have 
\begin{align*}
\Psi'(x) & =-\Psi(x)\Phi_{t}'(x)\Psi(x)\\
 & =-\Psi(x)P_{1}^{-1}(\i t\mathcal{H}(x)^{-1}-P_{0})\Phi_{t}(x)\Psi(x)\\
 & =-\Psi(x)P_{1}^{-1}(\i t\mathcal{H}(x)^{-1}-P_{0}).
\end{align*}
Hence, 
\begin{align*}
\left(\Psi^{\ast}\right)'(x)=\Psi'(x)^{\ast} & =(\i t\mathcal{H}(x)^{-1}-P_{0})P_{1}^{-1}\Psi(x)^{\ast}
\end{align*}
and 
\[
\left(P_{1}^{-1}\Psi^{\ast}\right)'(x)=P_{1}^{-1}(\i t\mathcal{H}(x)^{-1}-P_{0})P_{1}^{-1}\Psi(x)^{\ast}.
\]
Thus, $P_{1}^{-1}\Psi^{*}$ solves the same ODE as $\Phi_{t}$ does;
taking into account the initial values it follows that 
\[
P_{1}^{-1}\Psi^{\ast}(x)=\Phi_{t}(x)P_{1}^{-1}
\]
and hence, 
\[
\Phi_{t}(x)^{-1}=P_{1}^{-1}\Phi_{t}(x)^{\ast}P_{1}.\tag*{\qedhere}
\]
 
\end{proof}
The theorem underlying the proof of our characterisation of exponential
stability is the celebrated result by Gearhart--Prüß.
\begin{thm}[Gearhart--Prüß, see \cite{Pruess1984}]
\label{thm:GPT} $\mathcal{A}$ generates an exponentially stable
$C_{0}$-semi-group on a Hilbert space $H$, if, and only if, 
\begin{equation}
\sup_{z\in\mathbb{C}_{\Re>0}}\|(z-\mathcal{A})^{-1}\|_{L(H)}<\infty.\label{eq:exp_stable1}
\end{equation}
\end{thm}

The decisive step to reach our goal now becomes the following result.
\begin{thm}
\label{thm:main-1}Let $\mathcal{A}$ be as in \prettyref{thm:m-accretive-pH}
(iii) and assume that $\sup_{t\in\R}\|\Phi_{t}\|_{\infty}<\infty.$
Then the following statements are equivalent:

\begin{enumerate}[(i)]

\item $\i\R\subseteq\rho(\mathcal{A})$ and $\sup_{t\in\R}\|(\i t-\mathcal{A})^{-1}\|<\infty$,

\item for all $t\in\R$ the matrix 
\[
T_{t}\coloneqq W\left(\begin{array}{c}
\Phi_{t}(b)\\
1
\end{array}\right)\in\R^{d\times d}
\]
is invertible and $\sup_{t\in\R}\|T_{t}^{-1}\|<\infty.$

\end{enumerate}
\end{thm}

\begin{proof}
We use the operator $S$ as given in \prettyref{lem:basictransform}.
First note that for $f\in H$ we find $u\in\dom(\mathcal{A})$ with
$(\i t-\mathcal{A})u=f$ if and only if 
\[
S^{\ast}(\i t-\mathcal{A})SS^{-1}u=S^{\ast}f
\]
and 
\[
S^{\ast}(\i t-\mathcal{A})S=(\i tS^{\ast}S-S^{\ast}\mathcal{A}S)=(\i t\mathcal{H}^{-1}-A),
\]
where 
\[
A\colon\dom(A)\subseteq L_{2}([a,b];\R^{d})\to L_{2}([a,b];\R^{d}),\quad v\mapsto P_{1}v'+P_{0}v
\]
and 
\[
\dom(A)=\left\{ v\in H^{1}([a,b];\R^{d})\,;\:W\left(\begin{array}{c}
v(b)\\
v(a)
\end{array}\right)=0\right\} .
\]
Now $(\i t-\mathcal{A})u=f$ if and only if $v\coloneqq S^{-1}u\in\dom(A)$
satisfies 
\[
\i t\mathcal{H}^{-1}v-P_{1}v'-P_{0}v=S^{\ast}f=f,
\]
which in turn is equivalent to 
\[
v'=P_{1}^{-1}(\i t\mathcal{H}^{-1}-P_{0})v-P_{1}^{-1}f.
\]
Note that then 
\[
v(x)=\Phi_{t}(x)v_{0}-\Phi_{t}(x)\int_{a}^{x}\Phi_{t}(s)^{-1}P_{1}^{-1}f(s)\d s
\]
for some $v_{0}=v(a)\in\R^{d}.$ Then $v\in\dom(A)$ is equivalent
to 
\begin{align*}
0 & =W\left(\begin{array}{c}
v(b)\\
v(a)
\end{array}\right)\\
 & =W\left(\begin{array}{c}
\Phi_{t}(b)\\
1
\end{array}\right)v_{0}-W\left(\begin{array}{c}
\Phi_{t}(b)\int_{a}^{b}\Phi_{t}(s)^{-1}P_{1}^{-1}f(s)\d s\\
0
\end{array}\right)\\
 & =T_{t}v_{0}-W\left(\begin{array}{c}
\Phi_{t}(b)\int_{a}^{b}\Phi_{t}(s)^{-1}P_{1}^{-1}f(s)\d s\\
0
\end{array}\right).
\end{align*}

(i) $\Rightarrow$ (ii): We first show that $T_{t}$ is invertible.
For this, let $v_{0}\in\R^{d}$ with $T_{t}v_{0}=0.$ Set $v(x)\coloneqq\Phi(x)v_{0}.$
Then by what we have shown above $u\coloneqq Sv\in\dom(\mathcal{A})$
and satisfies $(\i t-\mathcal{A})u=0$ and thus, $u=0.$ The latter
gives $v=0$ and thus, $v(a)=v_{0}=0.$ So $T_{t}$ is invertible.
Assume now that 
\[
\sup_{t\in\R}\|T_{t}^{-1}\|=\infty.
\]
Since $t\mapsto T_{t}^{-1}$ is bounded on compact sets, we find a
sequence $(t_{n})_{n}$ with $t_{n}\to\infty$ such that $\|T_{t_{n}}^{-1}\|\to\infty.$
By the uniform boundedness principle we find $z\in\R^{d}$ such that
\[
\|T_{t_{n}}^{-1}z\|\to\infty\quad(n\to\infty).
\]
Since $W$ has full rank, we find $y,\tilde{y}\in\R^{d}$ such that
$z=W\left(\begin{array}{c}
y\\
\tilde{y}
\end{array}\right).$ Set now $y_{n}\coloneqq-y+\Phi_{t_{n}}(b)\tilde{y}$ and note that
$(y_{n})_{n}$ is bounded, since $(\Phi_{t_{n}}(b))_{n}$ is bounded.
Further, set 
\[
f_{n}(x)\coloneqq-\frac{1}{b-a}P_{1}\Phi_{t_{n}}(x)\Phi_{t_{n}}(b)^{-1}y_{n}\quad(x\in[a,b]).
\]
Then $f_{n}\in L_{\infty}([a,b];\R^{d})\subseteq H$ and $(f_{n})_{n}$
is uniformly bounded in $L_{\infty}([a,b];\R^{d})$, where we use
\prettyref{lem:inverse_fundamental_matrix}. Set now $u_{n}\coloneqq(\i t_{n}-\mathcal{A})^{-1}f_{n}$
and note that $(u_{n})_{n}$ is uniformly bounded in $H$ by assumption.
Hence, so is $v_{n}\coloneqq S^{-1}u_{n}$ and moreover, 
\[
v_{n}(x)=\Phi_{t_{n}}(x)v_{0,n}-\Phi_{t_{n}}(x)\int_{a}^{x}\Phi_{t_{n}}(s)^{-1}P_{1}^{-1}f_{n}(s)\d s
\]
 with 
\begin{align*}
T_{t_{n}}v_{0,n} & =W\left(\begin{array}{c}
\Phi_{t}(b)\int_{a}^{b}\Phi_{t}(s)^{-1}P_{1}^{-1}f_{n}(s)\d s\\
0
\end{array}\right)=-W\left(\begin{array}{c}
y_{n}\\
0
\end{array}\right)\\
 & =W\left(\begin{array}{c}
y-\Phi_{t_{n}}(b)\tilde{y}\\
0
\end{array}\right)=W\left(\begin{array}{c}
y\\
\tilde{y}
\end{array}\right)-T_{t_{n}}\tilde{y}=z-T_{t_{n}}\tilde{y}.
\end{align*}
The latter gives 
\[
T_{t_{n}}^{-1}z=v_{0,n}+\tilde{y}.
\]
Hence $\|v_{0,n}\|\to\infty$, but on the other hand 
\[
\|v_{0,n}\|=\frac{1}{\sqrt{b-a}}\|v_{0,n}\|_{L_{2}([a,b];\R^{d})}\leq C\left(\|v_{n}\|_{L_{2}([a,b];\R^{d})}+\|f_{n}\|_{L_{2}([a,b];\R^{d})}\right),
\]
for some constant $C>0$, where we again invoke \prettyref{lem:inverse_fundamental_matrix}.
This yields a contradiction.\\
(ii) $\Rightarrow$(i): By our considerations at the beginning of
the proof, we obtain 
\[
(\i t-\mathcal{A})u=f
\]
if and only if $u=Sv$ for 
\[
v(x)=\Phi_{t}(x)v_{0}-\Phi_{t}(x)\int_{a}^{x}\Phi_{t}(s)^{-1}P_{1}^{-1}f(s)\d s
\]
with 
\[
T_{t}v_{0}=W\left(\begin{array}{c}
\Phi_{t}(b)\int_{a}^{b}\Phi_{t}(s)^{-1}P_{1}^{-1}f(s)\d s\\
0
\end{array}\right).
\]
Since the last equality can be solved uniquely, since $T_{t}$ is
invertible, we infer that $\i t\in\rho(\mathcal{A})$ and by \prettyref{lem:inverse_fundamental_matrix}
we have 
\[
\|(\i t-\mathcal{A})^{-1}f\|_{H}\leq\|S\|\|v\|_{L_{2}}\leq C\left(\|v_{0}\|+\|f\|_{L_{2}}\right)\leq\tilde{C}\|f\|_{H},
\]
where we have used that $\|\cdot\|_{H}$ and $\|\cdot\|_{L_{2}}$
are equivalent and $T_{t}^{-1}$ is uniformly bounded. 
\end{proof}
Our main result may now be stated as follows.
\begin{thm}
\label{thm:mrcompressed}Let $P_{1},P_{0}\in\R^{d\times d}$ such
that $P_{1}=P_{1}^{\ast}$ is invertible, $P_{0}^{*}=-P_{0}$, and
$\mathcal{H}:[a,b]\to\R^{d\times d}$ be a measurable function such
that
\[
\exists m,M>0\:\forall x\in[a,b]:\:m\leq\mathcal{H}(x)=\mathcal{H}(x)^{\ast}\le M.
\]
We consider $\mathcal{A}\subseteq P_{1}\partial_{x}\mathcal{H}+P_{0}\mathcal{H}$
on the Hilbert space \textup{$H\coloneqq(L_{2}(a,b)^{d},\langle\mathcal{H}\cdot,\cdot\rangle_{L_{2}})$
}with domain 
\[
\dom(\mathcal{A})=\left\{ u\in H\,;\,\mathcal{H}u\in H^{1}([a,b])^{d},\,W\left(\begin{array}{c}
\left(\mathcal{H}u\right)(b)\\
\left(\mathcal{H}u\right)(a)
\end{array}\right)=0\right\} 
\]
for a matrix $W\in\R^{d\times2d}$ satisfying $\rk W=d$ and 
\[
W\left(\begin{array}{cc}
P_{1} & -P_{1}\\
1_{d} & 1_{d}
\end{array}\right)^{-1}\left(\begin{array}{cc}
0 & 1_{d}\\
1_{d} & 0
\end{array}\right)\left(W\left(\begin{array}{cc}
P_{1} & -P_{1}\\
1_{d} & 1_{d}
\end{array}\right)^{-1}\right)^{\ast}\geq0.
\]
Moreover, assume that the fundamental matrix $\Phi_{t}$ associated
to the ODE-system 
\[
u'(x)=P_{1}^{-1}(\i t\mathcal{H}(x)^{-1}-P_{0})u(x)\quad(x\in[a,b])
\]
with $\Phi_{t}(a)=1_{d}$ for each $t\in\R$ satisfies $\sup_{t\in\R}\|\Phi_{t}\|<\infty.$
Then the following statements are equivalent:

\begin{enumerate}[(i)]

\item the (contraction-)semi-group $\left(\e^{t\mathcal{A}}\right)_{t\geq0}$
is exponentially stable,

\item for each $t\in\R$ the matrix 
\[
T_{t}\coloneqq W\left(\begin{array}{c}
\Phi_{t}(b)\\
I_{d}
\end{array}\right)
\]
is invertible with $\sup_{t\in\R}\|T_{t}^{-1}\|<\infty.$

\end{enumerate}
\end{thm}

\begin{proof}
The assumptions guarantee that $\mathcal{A}$ is $m$-dissipative
by \prettyref{thm:m-accretive-pH}. Thus, using \prettyref{thm:GPT},
(ii) holds, if and only if, condition (i) in \prettyref{thm:main-1}
is satisfied. Thus, the claim follows from \prettyref{thm:main-1}.
\end{proof}

\section{A First Application\label{sec:FirstApplications}}

In this section, we specialise to the case $P_{0}=0$. Thus, throughout
this section, let
\begin{align*}
\mathcal{A}\colon\dom(\mathcal{A})\subseteq H & \to H\\
u & \mapsto P_{1}(\mathcal{H}u)',
\end{align*}
be a port-Hamiltonian operator as in \prettyref{eq:pHo} with $P_{0}=0$.
Furthermore, we assume $\mathcal{A}$ to generate a $C_{0}$-semi-group
of contractions. Hence, by \prettyref{thm:GenContr} we find a full
rank matrix $W\in\mathbb{R}^{d\times2d}$ satisfying \prettyref{eq:posdefWB}
such that
\[
\dom(\mathcal{A})=\left\{ u\in H\,;\,\mathcal{H}u\in H^{1}(a,b)^{d},\,W\left(\begin{array}{c}
\left(\mathcal{H}u\right)(b)\\
\left(\mathcal{H}u\right)(a)
\end{array}\right)=0\right\} .
\]
We recall $\Phi_{t}$, the fundamental solution associated with 
\[
u'(x)=\i tP_{1}^{-1}\mathcal{H}(x)^{-1}u(x)\quad(x\in(a,b))
\]
such that $\Phi_{t}(a)=1_{d}$ and condition \ref{eq:(B)} stating
$\sup_{t\in\mathbb{R}}\|\Phi_{t}\|_{\infty}<\infty$.

We specialise this result to the case when the Hamiltonian density
$\mathcal{H}$ is constant. For this we state a special case of \prettyref{thm:HBV},
which we prove in \prettyref{subsec:A-regularity-condition}.
\begin{prop}
\label{prop:constantH_0}Let $\mathcal{H}_{0}=\mathcal{H}_{0}^{*}\in\mathbb{R}^{d\times d}$
be non-negative and invertible. Then for all $Q_{1}=Q_{1}^{*}\in\mathbb{R}^{d\times d}$
invertible
\[
\sup_{t\in\mathbb{R}}\|\e^{\i tQ_{1}\mathcal{H}_{0}}\|<\infty.
\]
\end{prop}

\begin{cor}
\label{cor:constantH}Assume, additionally to the assumptions in this
section, $\mathcal{H}(x)=\mathcal{H}_{0}$ for some $\mathcal{H}_{0}\in\mathbb{R}^{d\times d}$
and all $x\in(a,b)$. Then the following conditions are equivalent:

\begin{enumerate}[(i)]

\item $\mathcal{A}$ generates an exponentially stable $C_{0}$-semi-group.

\item For all $t\in\R$ 
\[
T_{t}\coloneqq W\left(\begin{array}{c}
\e^{\i tP_{1}^{-1}\mathcal{H}_{0}^{-1}}\\
1_{d}
\end{array}\right)
\]
is invertible with $\sup_{t\in\R}\|T_{t}^{-1}\|<\infty.$

\end{enumerate}
\end{cor}

\begin{proof}
Since $\mathcal{H}$ is constant equal $\mathcal{H}_{0}$, $\Phi_{t}(x)=\e^{\i t(x-a)P_{1}^{-1}\mathcal{H}_{0}^{-1}}$.
Thus, by \prettyref{prop:constantH_0}, \ref{eq:(B)} holds, which
in turn leads to the applicability of \prettyref{thm:mrcompressed}.
To obtain the claim it thus suffices to note that $t$ runs through
all reals, making the prefactor $(b-a)$ superfluous, and to read
off the equivalence stated in \prettyref{thm:mrcompressed}.
\end{proof}
Before we put this characterisation result into perspective of the
results available in the literature, we provide an example that confirms
that invertibility for $T_{t}$ alone is not sufficient to deduce
the uniform bound of the inverses.
\begin{example}
\label{exa:unifiormsmatter}Consider the matrices $M\coloneqq-\frac{1}{2}\left(\begin{array}{cc}
1 & 1\\
1 & 1
\end{array}\right)$ (note that $\|M\|=1$) and $\mathcal{H}(x)\coloneqq\mathcal{H}_{0}\coloneqq\left(\begin{array}{cc}
1 & 0\\
0 & \sqrt{2}
\end{array}\right)^{-1}$ for all $x\in(a,b)\coloneqq(0,1)$ as well as $P_{1}=-1_{2}=\left(\begin{array}{cc}
-1 & 0\\
0 & -1
\end{array}\right)$. By \prettyref{thm:m-accretive-pH} (iv)
\[
W=\left(\begin{array}{cc}
-M & 1_{2}\end{array}\right)
\]
leads to $\mathcal{A}$ generate a contraction semi-group. We consider
\begin{align*}
T_{t} & =W\left(\begin{array}{c}
\e^{\i tP_{1}^{-1}\mathcal{H}_{0}^{-1}}\\
1_{2}
\end{array}\right)=\left(\begin{array}{cc}
-M & 1_{2}\end{array}\right)\left(\begin{array}{c}
\e^{-\i t\mathcal{H}_{0}^{-1}}\\
1_{2}
\end{array}\right)=-M\e^{-\i t\mathcal{H}_{0}^{-1}}+1_{2}
\end{align*}
and analyse the respective inverses for all $t\in\R$. We first show
that $T_{t}$ is invertible. For this, we compute
\begin{align*}
\det(T_{t}) & =\det\left(\left(\begin{array}{cc}
1 & 0\\
0 & 1
\end{array}\right)+\frac{1}{2}\left(\begin{array}{cc}
1 & 1\\
1 & 1
\end{array}\right)\left(\begin{array}{cc}
\e^{-\i t} & 0\\
0 & \e^{-\i\sqrt{2}t}
\end{array}\right)\right)=\det\left(\begin{array}{cc}
1+\frac{1}{2}\e^{-\i t} & \frac{1}{2}\e^{-\i\sqrt{2}t}\\
\frac{1}{2}\e^{-\i t} & 1+\frac{1}{2}\e^{-\i\sqrt{2}t}
\end{array}\right)\\
 & =1+\tfrac{1}{4}\e^{-\i(1+\sqrt{2})t}+\tfrac{1}{2}\e^{-\i t}+\tfrac{1}{2}\e^{-\i\sqrt{2}t}-\tfrac{1}{4}\e^{-\i(1+\sqrt{2})t}=1+\tfrac{1}{2}\left(\e^{-\i t}+\e^{-\i\sqrt{2}t}\right).
\end{align*}
Since $|\e^{-\i t}|=|\e^{-\i\sqrt{2}t}|=1,$ the determinant vanishes
if and only if $\e^{-\i t}=\e^{-\i\sqrt{2}t}=-1.$ However, since
$\sqrt{2}$ is irrational, this cannot happen for any $t\in\R.$ Hence,
$T_{t}$ is invertible for all $t\in\R$. Moreover, choosing $t_{k}\coloneqq-3\pi k$
for $k\in\Z,$ we obtain 
\[
\e^{-\i t_{k}}=-1,\quad\e^{-\i\sqrt{2}t_{k}}=\e^{\i\sqrt{2}\pi}\e^{\i2\pi\sqrt{2}k}
\]
and since $\{\e^{\i2\pi\sqrt{2}k}\,;\,k\in\Z\}$ lies dense in the
unit sphere $S_{1}$ (see, e.g., \cite[Theorem 3.13]{Devaney1989})
the net $(\det(T_{t_{k}}))_{k\in\Z}$ accumulates at $0$. Since $T_{t}$
itself is bounded in $t$, it follows that $(T_{t}^{-1})_{t\in\R}$
is unbounded. 
\end{example}

Next, one could ask whether or not the exponential stability of the
semi-group generated by $\mathcal{A}$ depends on the Hamiltonian
density $\mathcal{H}$. \cite[Section 5]{Engel2013} provided an example
confirming dependence. We reprove this result with the characterisation
from above.
\begin{example}
Let $\theta\in\mathbb{R}$, $\theta>-1$. Consider the matrices $M\coloneqq\frac{1}{2}\left(\begin{array}{cc}
1 & 1\\
-1 & -1
\end{array}\right)$ (note that $\|M\|=1$) and $\mathcal{H}(x)\coloneqq\mathcal{H}_{\theta}\coloneqq\left(\begin{array}{cc}
1+\theta & 0\\
0 & 1
\end{array}\right)$ for all $x\in(a,b)\coloneqq(0,1)$ as well as $P_{1}=1_{2}$. By
\prettyref{thm:m-accretive-pH} (iv)
\[
W=\left(\begin{array}{cc}
1_{2} & -M\end{array}\right)
\]
leads to $\mathcal{A}_{\theta}\coloneqq P_{1}\partial_{x}\mathcal{H}_{\theta}=\partial_{x}\mathcal{H}_{\theta}$
generate a contraction semi-group. In order to assess for which $\theta$
this semi-group is exponentially stable, using \prettyref{cor:constantH},
we consider
\begin{align*}
T_{t}^{\theta} & \coloneqq W\left(\begin{array}{c}
\e^{\i t(b-a)P_{1}^{-1}\mathcal{H}_{\theta}^{-1}}\\
1_{2}
\end{array}\right)=\left(\begin{array}{cc}
1_{2} & -M\end{array}\right)\left(\begin{array}{c}
\e^{\i t\mathcal{H}_{\theta}^{-1}}\\
1_{2}
\end{array}\right)\\
 & =\e^{\i t\mathcal{H}_{\theta}^{-1}}-M=\left(\begin{array}{cc}
\e^{\i\frac{t}{1+\theta}} & 0\\
0 & \e^{\i t}
\end{array}\right)-\frac{1}{2}\left(\begin{array}{cc}
1 & 1\\
-1 & -1
\end{array}\right)
\end{align*}
and compute its determinant by
\[
\det T_{t}^{\theta}=\frac{1}{2}\e^{\i\frac{t}{1+\theta}}-\frac{1}{2}\e^{\i t}+\e^{\i t\frac{2+\theta}{1+\theta}}.
\]
Since $T_{t}^{\theta}$ is uniformly bounded in $t,$ the uniform
boundedness of $\left(T_{t}^{\theta}\right)^{-1}$ is equivalent to
the uniform boundedness of $\frac{1}{\det T_{t}^{\theta}}$. For $\theta=0$,
we obtain that $T_{t}^{0}$ is invertible with bound for the inverse
uniform in $t$, leading to $\mathcal{A}_{0}$ generating an exponentially
stable semi-group. If $t=3\pi$ and $\theta=1/2$, then $T_{t}^{\theta}$
is not invertible and, hence, $\mathcal{A}_{1/2}$ does not generate
an exponentially stable semi-group. A closer look at the proof of
\prettyref{thm:main-1} reveals that $3\pi\i\in\sigma_{p}(\mathcal{A}_{1/2})$
and hence, the generated semi-group is not even asymptotically stable. 
\end{example}

\section{A Sufficient Criterion for Exponential Stability\label{sec:A-Sufficient-Criterion}}

It is the aim of this section to put the characterisation into perspective
of the literature. For this, throughout this section, we let
\begin{align*}
\mathcal{A}\colon\dom(\mathcal{A})\subseteq H & \to H\\
u & \mapsto P_{1}(\mathcal{H}u)'+P_{0}\mathcal{H}u,
\end{align*}
be a port-Hamiltonian operator as in \prettyref{eq:pHo}. We recall
$\Phi_{t}$, the fundamental solution associated with 
\[
u'(x)=P_{1}^{-1}(\i t\mathcal{H}(x)^{-1}-P_{0})u(x)\quad(x\in(a,b))
\]
such that $\Phi_{t}(a)=1_{d}$ and condition \ref{eq:(B)} stating
$\sup_{t\in\mathbb{R}}\|\Phi_{t}\|_{\infty}<\infty$. 
\begin{thm}
\label{thm:suffcrit-1} Assume \ref{eq:(B)} and that $\mathcal{A}$
generates a contraction semi-group. If there exist $c>0$ and $\eta\in\{a,b\}$
such that for all $u\in\dom(\mathcal{A})$ 
\begin{equation}
\langle u,\mathcal{A}u\rangle_{H}\leq-c\|\mathcal{H}u(\eta)\|^{2},\label{eq:posdefAsuff}
\end{equation}
then $\mathcal{A}$ generates an exponentially stable $C_{0}$-semi-group.
\end{thm}

\begin{rem}
In \prettyref{subsec:A-regularity-condition}, we confirm that \prettyref{thm:suffcrit-1}
implies \prettyref{thm:suffcrit}. In fact, we shall provide criteria
warranting \ref{eq:(B)} to be satisfied. One of them being $\mathcal{H}$
to be of bounded variation, see also \prettyref{sec:CondB}. Another
criterion requires a structural hypothesis on the interplay of $\mathcal{H}$
and $P_{1}$, which is independent of regularity, thus, showing that
the statements in \prettyref{thm:suffcrit-1} and \prettyref{thm:suffcrit}
are not equivalent, eventually proving that \prettyref{thm:suffcrit-1}
is a proper generalisation of \prettyref{thm:suffcrit}.
\end{rem}

\begin{lem}
\label{lem:trace_onto}Let $\mathcal{A}$ be as in \prettyref{thm:m-accretive-pH}.
Then the mapping 
\[
\tr\colon\dom(\mathcal{A})\to\R^{d}\quad u\mapsto Q_{-}(\mathcal{H}u)(b)+Q_{+}(\mathcal{H}u)(a)
\]
is onto.
\end{lem}

\begin{proof}
Let $y\in\R^{d}$ and $M$ as in \prettyref{thm:m-accretive-pH} (iv).
We define a function $v_{+}\colon[a,b]\to E_{+}$ by 
\[
v_{+}(t)\coloneqq\frac{t-a}{b-a}\left(P_{1}^{+}\right)^{-\frac{1}{2}}\iota_{+}^{\ast}My+\frac{t-b}{a-b}\left(P_{1}^{+}\right)^{-\frac{1}{2}}\iota_{+}^{\ast}y
\]
and similarly $v_{-}\colon[a,b]\to E_{-}$ by 
\[
v_{-}(t)\coloneqq\frac{t-a}{b-a}\left(P_{1}^{-}\right)^{-\frac{1}{2}}\iota_{-}^{\ast}y+\frac{t-b}{a-b}\left(P_{1}^{-}\right)^{-\frac{1}{2}}\iota_{-}^{\ast}My.
\]
Then clearly, $v\coloneqq\iota_{+}v_{+}+\iota_{-}v_{-}\in H^{1}([a,b];\R^{d})$
and it satisfies 
\begin{align*}
Q_{-}v(b)+Q_{+}v(a) & =\iota_{-}\iota_{-}^{\ast}y+\iota_{+}\iota_{+}^{\ast}y=y,\\
Q_{-}v(a)+Q_{+}v(b) & =\iota_{-}\iota_{-}^{\ast}My+\iota_{+}\iota_{+}^{\ast}My=My,
\end{align*}
which shows on the one hand that $u\coloneqq\mathcal{H}^{-1}v\in\dom(\mathcal{A})$
and on the other hand that $\tr u=y$.
\end{proof}
It is possible to recast the condition in \prettyref{thm:suffcrit-1}
as an inequality merely containing finite-dimensional spaces. This
is a simple albeit decisive observation for the proof of \prettyref{thm:suffcrit-1}.
\begin{lem}
\label{lem:equiv_sufficient_cond}Let $\mathcal{A}$ be as in \prettyref{thm:m-accretive-pH}
with $M$ as in \prettyref{thm:m-accretive-pH} (iv). Moreover let
$c>0$ and $\eta\in\{a,b\}.$ Then the following statements are equivalent:

\begin{enumerate}[(a)] 

\item For all $u\in\dom(\mathcal{A})$ 
\[
\langle\mathcal{A}u,u\rangle_{H}\leq-c\|\left(\mathcal{H}u\right)(\eta)\|^{2}.
\]

\item For all $y\in\R^{d}$ either 
\[
\|y\|^{2}-\|My\|^{2}\geq2c\left(\|\iota_{-}^{\ast}My\|_{E_{-}}^{2}+\|\iota_{+}^{\ast}y\|_{E_{+}}^{2}\right)
\]
if $\eta=a$ or 
\[
\|y\|^{2}-\|My\|^{2}\geq2c\left(\|\iota_{-}^{\ast}y\|_{E_{-}}^{2}+\|\iota_{+}^{\ast}My\|_{E_{+}}^{2}\right)
\]
if $\eta=b.$

\end{enumerate}
\[
\]
\end{lem}

\begin{proof}
By \prettyref{thm:m-accretive-pH} we have 
\begin{align*}
\langle\mathcal{A}u,u\rangle_{H} & =\frac{1}{2}\left(\|M\left(Q_{-}(\mathcal{H}u)(b)+Q_{+}(\mathcal{H}u)(a)\right)\|^{2}-\|Q_{-}\left(\mathcal{H}u\right)(b)+Q_{+}\left(\mathcal{H}u\right)(a)\|^{2}\right)\\
 & =\frac{1}{2}\left(\|M\tr u\|^{2}-\|\tr u\|^{2}\right)
\end{align*}
for all $u\in\dom(\mathcal{A}).$ Moreover, 
\begin{align*}
\left(\mathcal{H}u\right)(a) & =\iota_{+}(P_{1}^{+})^{-\frac{1}{2}}\iota_{+}^{\ast}Q_{+}\left(\mathcal{H}u\right)(a)+\iota_{-}(P_{1}^{-})^{-\frac{1}{2}}\iota_{-}^{\ast}Q_{-}\left(\mathcal{H}u\right)(a)\\
 & =\iota_{+}(P_{1}^{+})^{-\frac{1}{2}}\iota_{+}^{\ast}\tr u+\iota_{-}(P_{1}^{-})^{-\frac{1}{2}}\iota_{-}^{\ast}M\tr u,\\
\left(\mathcal{H}u\right)(b) & =\iota_{+}(P_{1}^{+})^{-\frac{1}{2}}\iota_{+}^{\ast}Q_{+}\left(\mathcal{H}u\right)(b)+\iota_{-}(P_{1}^{-})^{-\frac{1}{2}}\iota_{-}^{\ast}Q_{-}\left(\mathcal{H}u\right)(b)\\
 & =\iota_{+}(P_{1}^{+})^{-\frac{1}{2}}\iota_{+}^{\ast}M\tr u+\iota_{-}(P_{1}^{-})^{-\frac{1}{2}}\iota_{-}^{\ast}\tr u.
\end{align*}
Now the assertion follows from \prettyref{lem:trace_onto}.
\end{proof}
\begin{prop}
\label{prop:fundamental_matrix}Let $\mathcal{O}\in L_{\infty}([a,b];\R^{d\times d})$
such that $\mathcal{O}(x)$ is self-adjoint for a.e. $x\in[a,b].$
Moreover, let $\Pi\colon[a,b]\to\R^{d\times d}$ be the fundamental
solution to the differential equation 
\[
v'(x)=P_{1}^{-1}(\i\mathcal{O}(x)-P_{0})v(x)
\]
with $\Pi(a)=1_{d}.$ Then the following statements hold:

\begin{enumerate}[(a)]

\item The matrix 
\[
V\coloneqq Q_{-}+Q_{+}\Pi(b)
\]
is invertible with $\|V^{-1}\|\leq C_{1}\|\Pi(b)\|+C_{2},$ where
$C_{1},C_{2}$ are just depending on the values of $P_{1}$. Moreover,
the matrix 
\[
U\coloneqq\left(Q_{-}\Pi(b)+Q_{+}\right)V^{-1}
\]
is unitary.

\item For $y\in\R^{d}$ we have the following estimates
\[
\|y\|^{2}\leq\|V\|^{2}\left(\|\iota_{+}^{\ast}Uy\|_{E_{+}}^{2}+\|\iota_{-}^{\ast}y\|_{E_{-}}^{2}\right)
\]
and 
\[
\|y\|^{2}\leq\|\Pi(b)^{-1}\|^{2}\|V\|^{2}\left(\|\iota_{+}^{\ast}y\|_{E_{+}}^{2}+\|\iota_{-}^{\ast}Uy\|_{E_{-}}^{2}\right).
\]

\end{enumerate}
\end{prop}

\begin{proof}
(a) We recall from \prettyref{lem:inverse_fundamental_matrix} that
$\Pi(b)^{-1}=P_{1}^{-1}\Pi(b)^{\ast}P_{1}$ or equivalently $P_{1}=\Pi(b)^{\ast}P_{1}\Pi(b)$.
We set $W\coloneqq Q_{-}\Pi(b)+Q_{+}$ and compute 
\begin{align*}
V^{\ast}V & =(Q_{-}+\Pi(b)^{\ast}Q_{+})\left(Q_{-}+Q_{+}\Pi(b)\right)\\
 & =Q_{-}^{2}+\Pi(b)^{\ast}Q_{+}^{2}\Pi(b)\\
 & =Q_{+}^{2}-P_{1}+\Pi(b)^{\ast}\left(P_{1}+Q_{-}^{2}\right)\Pi(b)\\
 & =Q_{+}^{2}+\Pi(b)^{\ast}Q_{-}^{2}\Pi(b)\\
 & =W^{\ast}W.
\end{align*}
In particular, we have $\|Vx\|=\|Wx\|$ for each $x\in\R^{d}$. Thus,
if $Vx=0,$ then $Wx=0$ and hence, $Q_{-}x=Q_{-}Vx=0$ as well as
$Q_{+}x=Q_{+}Wx=0.$ The latter gives $P_{1}x=(Q_{+}^{2}-Q_{-}^{2})x=0$
and thus, $x=0$ showing the invertibility of $V.$ Moreover, for
$x\in\R^{d}$ we compute 
\begin{align*}
\|Ux\|^{2} & =\langle WV^{-1}x,WV^{-1}x\rangle\\
 & =\langle V^{-1}x,W^{\ast}WV^{-1}x\rangle\\
 & =\langle V^{-1}x,V^{\ast}x\rangle=\|x\|^{2},
\end{align*}
showing that $U$ is unitary. It remains to prove the estimate for
the norm of the inverse of $V$. We set $D\coloneqq\iota_{+}^{\ast}\Pi(b)\iota_{+}$
and $C\coloneqq\iota_{-}^{\ast}\Pi(b)\iota_{+}$ compute for $x\in E_{-}$
\begin{align*}
\|\left(P_{1}^{+}\right)^{\frac{1}{2}}x\|^{2} & =\langle P_{1}^{+}x,x\rangle=\langle\iota_{+}^{\ast}P_{1}\iota_{+}x,x\rangle\\
 & =\langle\iota_{+}^{\ast}\left(\Pi(b)^{\ast}P_{1}\Pi(b)\right)\iota_{+}x,x\rangle\\
 & =\langle P_{1}\Pi(b)\iota_{+}x,\Pi(b)\iota_{+}x\rangle\\
 & =\langle P_{1}(\iota_{+}\iota_{+}^{\ast}+\iota_{-}\iota_{-}^{\ast})\Pi(b)\iota_{+}x,\Pi(b)\iota_{+}x\rangle\\
 & =\langle P_{1}^{+}Dx,Dx\rangle-\langle P_{1}^{-}Cx,Cx\rangle\\
 & \leq\|\left(P_{1}^{+}\right)^{\frac{1}{2}}Dx\|^{2}\leq\|\left(P_{1}^{+}\right)^{\frac{1}{2}}\|^{2}\|Dx\|^{2}.
\end{align*}
Hence, $D$ is invertible with $\|D^{-1}\|\leq\frac{\|\left(P_{1}^{+}\right)^{-\frac{1}{2}}\|}{\|\left(P_{1}^{+}\right)^{\frac{1}{2}}\|}$.
Since $V$ is unitarily equivalent (via the decomposition $\R^{d}=E_{+}\oplus E_{-})$
to the matrix 
\[
\left(\begin{array}{cc}
\left(P_{1}^{+}\right)^{\frac{1}{2}}D & \left(P_{1}^{+}\right)^{\frac{1}{2}}B\\
0 & \left(P_{1}^{-}\right)^{\frac{1}{2}}
\end{array}\right)
\]
with $B\coloneqq\iota_{+}^{\ast}\Pi(b)\iota_{-}$ , its inverse is
unitarily equivalent to 
\[
\left(\begin{array}{cc}
D^{-1}\left(P_{1}^{+}\right)^{-\frac{1}{2}} & -D^{-1}B\left(P_{1}^{-}\right)^{-\frac{1}{2}}\\
0 & \left(P_{1}^{-}\right)^{-\frac{1}{2}}
\end{array}\right)
\]
and thus, the desired estimate for $\|V^{-1}\|$ follows.%
{} 

(b) Let $y\in\R^{d}$. We compute 
\begin{align*}
\|\iota_{+}^{\ast}Uy\|_{E_{+}}^{2} & =\|\iota_{+}\left(P_{1}^{+}\right)^{-\frac{1}{2}}\iota_{+}^{\ast}Uy\|^{2}\\
 & =\|\iota_{+}\iota_{+}^{\ast}V^{-1}y\|^{2}
\end{align*}
and 
\begin{align*}
\|\iota_{-}^{\ast}y\|_{E_{-}}^{2} & =\|\iota_{-}(P_{1}^{-})^{-\frac{1}{2}}\iota_{-}^{\ast}y\|^{2}\\
 & =\|\iota_{-}(P_{1}^{-})^{-\frac{1}{2}}\iota_{-}^{\ast}VV^{-1}y\|^{2}\\
 & =\|\iota_{-}\iota_{-}^{\ast}V^{-1}y\|.
\end{align*}
Consequently, we obtain 
\[
\|\iota_{+}^{\ast}Uy\|_{E_{+}}^{2}+\|\iota_{-}^{\ast}y\|_{E_{-}}^{2}=\|V^{-1}y\|^{2}\geq\frac{1}{\|V\|^{2}}\|y\|^{2},
\]
which proves the first estimate. Similarly, we compute 
\begin{align*}
\|\iota_{-}^{\ast}Uy\|_{E_{-}}^{2} & =\|\iota_{-}\left(P_{1}^{-}\right)^{-\frac{1}{2}}\iota_{-}^{\ast}Uy\|^{2}\\
 & =\|\iota_{-}\iota_{-}^{\ast}\Pi(b)V^{-1}y\|^{2}
\end{align*}
as well as 
\begin{align*}
\|\iota_{+}^{\ast}y\|_{E_{+}}^{2} & =\|\iota_{+}(P_{1}^{+})^{-\frac{1}{2}}\iota_{+}^{\ast}y\|^{2}\\
 & =\|\iota_{+}(P_{1}^{+})^{-\frac{1}{2}}\iota_{+}^{\ast}VV^{-1}y\|^{2}\\
 & =\|\iota_{+}\iota_{+}^{\ast}\Pi(b)V^{-1}y\|^{2}.
\end{align*}
Hence, we obtain 
\[
\|\iota_{-}^{\ast}Uy\|_{E_{-}}^{2}+\|\iota_{+}^{\ast}y\|_{E_{+}}^{2}=\|\Pi(b)V^{-1}y\|^{2}\geq\frac{1}{\|\Pi(b)^{-1}\|^{2}\|V\|^{2}}\|y\|^{2}.\tag*{{\qedhere}}
\]
\end{proof}
\begin{proof}[Proof of \prettyref{thm:suffcrit-1}]
 By \prettyref{thm:mrcompressed} we need to prove that 
\[
T_{t}\coloneqq W\left(\begin{array}{c}
\Phi_{t}(b)\\
1
\end{array}\right)
\]
is invertible for each $t\in\R$ with $\sup_{t\in\R}\|T_{t}^{-1}\|<\infty.$
By \prettyref{thm:m-accretive-pH} the matrix $W$ can be expressed
by 
\[
W=K\left(\begin{array}{cc}
Q_{+}-MQ_{-} & Q_{-}-MQ_{+}\end{array}\right)
\]
with $\|M\|\leq1$ and $K$ invertible given as in \prettyref{thm:m-accretive-pH}
(iv). Hence, $T_{t}$ has the form 
\begin{align*}
T_{t} & =K\left(\left(Q_{+}-MQ_{-}\right)\Phi_{t}(b)+\left(Q_{-}-MQ_{+}\right)\right)\\
 & =K\left(Q_{-}+Q_{+}\Phi_{t}(b)-M(Q_{-}\Phi_{t}(b)+Q_{+})\right).
\end{align*}
As in \prettyref{prop:fundamental_matrix} (a) we set 
\[
V_{t}\coloneqq Q_{-}+Q_{+}\Phi_{t}(b)
\]
which is an invertible matrix by \prettyref{prop:fundamental_matrix}
(a) and 
\[
U_{t}\coloneqq(Q_{-}\Phi_{t}(b)+Q_{+})V_{t}^{-1}
\]
is unitary by \prettyref{prop:fundamental_matrix} (a). Let $y\in\R^{d}$.
We prove the assertion by showing that there exists some $\kappa>0$
independently of $y$ with 
\[
\|y\|\leq\kappa\|T_{t}y\|\quad(t\in\R).
\]
Using the representation above, we have 
\begin{align*}
T_{t}y & =K\left(V_{t}y-M(Q_{+}\Phi_{t}(b)+Q_{-})y\right)\\
 & =K(1-MU_{t})V_{t}y.
\end{align*}
Since $\|y\|=\|V_{t}^{-1}V_{t}y\|\leq\|V_{t}^{-1}\|\|V_{t}y\|$ and
$\sup_{t\in\R}\|V_{t}^{-1}\|<\infty$ by \prettyref{prop:fundamental_matrix}
(a), it suffices to prove 
\[
\|V_{t}y\|\leq\kappa\|\tilde{T}_{t}y\|\quad(t\in\R),
\]
where $\tilde{T}_{t}=K^{-1}T_{t}$. We employ \prettyref{lem:equiv_sufficient_cond}
to obtain
\[
\|y\|^{2}-\|My\|^{2}\geq2c\left(\|\iota_{-}^{\ast}My\|_{E_{-}}^{2}+\|\iota_{+}^{\ast}y\|_{E_{+}}^{2}\right)
\]
if $\eta=a$ or 
\[
\|y\|^{2}-\|My\|^{2}\geq2c\left(\|\iota_{-}^{\ast}y\|_{E_{-}}^{2}+\|\iota_{+}^{\ast}My\|_{E_{+}}^{2}\right)
\]
if $\eta=b.$ Let us start with the case $\eta=a.$ Applying this
inequality to $U_{t}V_{t}y$ and using that $V_{t}y=\tilde{T}_{t}y+MU_{t}V_{t}y$
as well as that $U_{t}$ is isometric, we get 
\begin{align*}
\|V_{t}y\|^{2} & =\|\tilde{T}_{t}y\|^{2}+2\langle MU_{t}V_{t}y,\tilde{T}_{t}y\rangle+\|MU_{t}V_{t}y\|^{2}\\
 & \leq\|\tilde{T}_{t}y\|^{2}+2\langle MU_{t}V_{t}y,\tilde{T}_{t}y\rangle+\|U_{t}V_{t}y\|^{2}-2c\left(\|\iota_{-}^{\ast}MU_{t}V_{t}y\|_{E_{-}}^{2}+\|\iota_{+}^{\ast}U_{t}V_{t}y\|_{E_{+}}^{2}\right)\\
 & =\|\tilde{T}_{t}y\|^{2}+2\langle V_{t}y-\tilde{T}_{t}y,\tilde{T}_{t}y\rangle+\|V_{t}y\|^{2}-2c\left(\|\iota_{-}^{\ast}V_{t}y-\iota_{-}^{\ast}\tilde{T}_{t}y\|_{E_{-}}^{2}+\|\iota_{+}^{\ast}U_{t}V_{t}y\|_{E_{+}}^{2}\right)\\
 & =-\|\tilde{T}_{t}y\|^{2}+2\langle V_{t}y,\tilde{T}_{t}y\rangle+\|V_{t}y\|^{2}-2c\left(\|\iota_{-}^{\ast}V_{t}y-\iota_{-}^{\ast}\tilde{T}_{t}y\|_{E_{-}}^{2}+\|\iota_{+}^{\ast}U_{t}V_{t}y\|_{E_{+}}^{2}\right).
\end{align*}
Thus, we have 
\begin{align*}
0 & \leq-\|\tilde{T}_{t}y\|^{2}+2\langle V_{t}y,\tilde{T}_{t}y\rangle-2c\left(\|\iota_{-}^{\ast}V_{t}y-\iota_{-}^{\ast}\tilde{T}_{t}y\|_{E_{-}}^{2}+\|\iota_{+}^{\ast}U_{t}V_{t}y\|_{E_{+}}^{2}\right)\\
 & =-\|\tilde{T}_{t}y\|^{2}+2\langle V_{t}y,\tilde{T}_{t}y\rangle-2c\left(\|\iota_{-}^{\ast}V_{t}y\|_{E_{-}}^{2}+\|\iota_{-}^{\ast}\tilde{T}_{t}y\|_{E_{-}}^{2}-2\langle\iota_{-}\left(P_{1}^{-}\right)^{-1}\iota_{-}^{\ast}\tilde{T}_{t}y,V_{t}y\rangle+\|\iota_{+}^{\ast}U_{t}V_{t}y\|_{E_{+}}^{2}\right)\\
 & \leq2\langle V_{t}y,\tilde{T}_{t}y\rangle-2c\left(\|\iota_{-}^{\ast}V_{t}y\|_{E_{-}}^{2}-2\langle\iota_{-}\left(P_{1}^{-}\right)^{-1}\iota_{-}^{\ast}\tilde{T}_{t}y,V_{t}y\rangle+\|\iota_{+}^{\ast}U_{t}V_{t}y\|_{E_{+}}^{2}\right),
\end{align*}
which yields 
\begin{align*}
\|\iota_{-}^{\ast}V_{t}y\|_{E_{-}}^{2}+\|\iota_{+}^{\ast}U_{t}V_{t}y\|_{E_{+}}^{2} & \leq\frac{1}{c}\langle V_{t}y,\tilde{T}_{t}y+2c\iota_{-}\left(P_{1}^{-}\right)^{-1}\iota_{-}^{\ast}\tilde{T}_{t}y\rangle\\
 & \le\frac{1}{2c}\left(\varepsilon\|V_{t}y\|^{2}+\frac{1}{\varepsilon}\|\tilde{T}_{t}y+2c\iota_{-}\left(P_{1}^{-}\right)^{-1}\iota_{-}^{\ast}\tilde{T}_{t}y\|{{}^2}\right)
\end{align*}
for each $\varepsilon>0.$ Invoking \prettyref{prop:fundamental_matrix}
(b), we have 
\[
\|V_{t}y\|^{2}\leq\|V_{t}\|^{2}\left(\|\iota_{-}^{\ast}V_{t}y\|_{E_{-}}^{2}+\|\iota_{+}^{\ast}U_{t}V_{t}y\|_{E_{+}}^{2}\right)\leq\frac{\|V_{t}\|^{2}}{2c}\left(\varepsilon\|V_{t}y\|^{2}+\frac{1}{\varepsilon}\|\tilde{T}_{t}y+2c\iota_{-}\left(P_{1}^{-}\right)^{-1}\iota_{-}^{\ast}\tilde{T}_{t}y\|{{}^2}\right).
\]
Hence, choosing $\varepsilon\coloneqq\frac{c}{\|V_{t}\|^{2}},$ we
derive 
\[
\|V_{t}y\|^{2}\leq\left(\frac{\|V_{t}\|^{2}}{c}\right)^{2}\|\tilde{T}_{t}y+2c\iota_{-}\left(P_{1}^{-}\right)^{-1}\iota_{-}^{\ast}\tilde{T}_{t}y\|{{}^2}\leq\kappa\|\tilde{T}_{t}y\|^{2},
\]
for some $\kappa>0$ independent of $t$ (note that $\sup_{t}\|V_{t}\|<\infty$).
This proves the assertion for the case $\eta=a.$ If $\eta=b,$ an
analogous computation gives
\[
\|V_{t}y\|^{2}\leq-\|\tilde{T}_{t}y\|^{2}+2\langle V_{t}y,\tilde{T}_{t}y\rangle+\|V_{t}y\|^{2}-2c\left(\|\iota_{+}^{\ast}V_{t}y-\iota_{+}^{\ast}\tilde{T}_{t}y\|_{E_{+}}^{2}+\|\iota_{-}^{\ast}U_{t}V_{t}y\|_{E_{-}}^{2}\right)
\]
and hence, 
\[
0\leq2\langle V_{t}y,\tilde{T}_{t}y\rangle-2c\left(\|\iota_{+}^{\ast}V_{t}y\|_{E_{+}}^{2}-2\langle\iota_{+}\left(P_{1}^{+}\right)^{-1}\iota_{+}^{\ast}\tilde{T}_{t}y,V_{t}y\rangle+\|\iota_{-}^{\ast}U_{t}V_{t}y\|_{E_{-}}^{2}\right).
\]
Thus, we infer 
\begin{align*}
\|\iota_{+}^{\ast}V_{t}y\|_{E_{+}}^{2}+\|\iota_{-}^{\ast}U_{t}V_{t}y\|_{E_{-}}^{2} & \leq\frac{1}{c}\langle V_{t}y,\tilde{T}_{t}y+2c\iota_{+}\left(P_{1}^{+}\right)^{-1}\iota_{+}^{\ast}\tilde{T}_{t}y\rangle\\
 & \leq\frac{1}{2c}\left(\varepsilon\|V_{t}y\|^{2}+\frac{1}{\varepsilon}\|\tilde{T}_{t}y+2c\iota_{+}\left(P_{1}^{+}\right)^{-1}\iota_{+}^{\ast}\tilde{T}_{t}y\|{{}^2}\right).
\end{align*}
Hence, involving the second estimate in \prettyref{prop:fundamental_matrix}
(b), we infer 
\[
\|V_{t}y\|^{2}\leq\frac{\|\Phi_{t}(b)^{-1}\|^{2}\|V_{t}\|^{2}}{2c}\left(\varepsilon\|V_{t}y\|^{2}+\frac{1}{\varepsilon}\|\tilde{T}_{t}y+2c\iota_{+}\left(P_{1}^{+}\right)^{-1}\iota_{+}^{\ast}\tilde{T}_{t}y\|{{}^2}\right)
\]
and choosing $\varepsilon\coloneqq\frac{c}{\|\Phi_{t}(b)^{-1}\|^{2}\|V_{t}\|^{2}},$
we end up with 
\[
\|V_{t}y\|_{\R^{d}}^{2}\leq\left(\frac{\|\Phi_{t}(b)^{-1}\|^{2}\|V_{t}\|^{2}}{c}\right)^{2}\|\tilde{T}_{t}y+2c\iota_{+}\left(P_{1}^{+}\right)^{-1}\iota_{+}^{\ast}\tilde{T}_{t}y\|{{}^2}\leq\tilde{\kappa}\|\tilde{T}_{t}y\|^{2},
\]
for some $\tilde{\kappa}>0$ independent of $t$ (note that $\sup_{t}\|\Phi_{t}(b)^{-1}\|<\infty$
by \prettyref{lem:inverse_fundamental_matrix}). 
\end{proof}
We obtain another sufficient condition for exponential stability.
\begin{thm}
Let $\mathcal{A}$ be as in \prettyref{thm:m-accretive-pH} and assume
\prettyref{eq:(B)}. Moreover assume that $W$ satisfies 
\[
W\left(\begin{array}{cc}
P_{1} & -P_{1}\\
1 & 1
\end{array}\right)^{-1}\left(\begin{array}{cc}
0 & 1\\
1 & 0
\end{array}\right)\left(W\left(\begin{array}{cc}
P_{1} & -P_{1}\\
1 & 1
\end{array}\right)^{-1}\right)^{\ast}>0.
\]
Then $\mathcal{A}$ generates an exponentially stable $C_{0}$-semi-group
on $H$. 
\end{thm}

\begin{proof}
Again, we need to prove that $T_{t}\coloneqq W\left(\begin{array}{c}
\Phi_{t}(b)\\
1
\end{array}\right)$ is invertible with $\sup_{t\in\R}\|T_{t}^{-1}\|<\infty.$ By \prettyref{lem:W_vs_M}
we find a matrix $M$ with $\|M\|<1$ and an invertible matrix $K$
such that 
\[
W=K\left(\begin{array}{cc}
Q_{+}-MQ_{-} & Q_{-}-MQ_{+}\end{array}\right)
\]
and thus, $T_{t}$ can be expressed as 
\[
T_{t}=K\left(Q_{+}\Phi_{t}(b)-MQ_{-}\Phi_{t}(b)+Q_{-}-MQ_{+}\right)=K\left(Q_{-}+Q_{+}\Phi_{t}(b)-M(Q_{-}\Phi_{t}(b)+Q_{+})\right).
\]
Using the matrices $V_{t}\coloneqq Q_{-}+Q_{+}\Phi_{t}(b)$ and $U_{t}\coloneqq(Q_{-}\Phi_{t}(b)+Q_{+})V_{t}^{-1}$,
we infer that $U_{t}$ is unitary and 
\[
T_{t}=K(1-MU_{t})V_{t}.
\]
Since $K$ and $V_{t}$ are both invertible and $\sup_{t}\|V_{t}^{-1}\|<\infty$
by \prettyref{prop:fundamental_matrix} (a), it suffices to show that
$1-MU_{t}$ is invertible and its inverse is uniformly bounded in
$t$. This, however, follows from the Neumann series, since $\|M\|<1$
and $\|U_{t}\|=1.$ Hence $(1-MU_{t})^{-1}=\sum_{k=0}^{\infty}(MU_{t})^{k}$
and 
\[
\|(1-MU_{t})^{-1}\|\leq\sum_{k=0}^{\infty}\|M\|^{k}=\frac{1}{1-\|M\|}.\tag*{\qedhere}
\]
\end{proof}

\section{The Condition \ref{eq:(B)}\label{sec:CondB}}

This section is devoted to a discussion of condition \ref{eq:(B)}.
For this, throughout this section, we let $a,b\in\R,$ $a<b$ and
\[
\mathcal{H}\colon\left[a,b\right]\to\R^{d\times d}\in L_{\infty}(a,b;\R^{d\times d})
\]
satisfying $\mathcal{H}(x)=\mathcal{H}(x)^{*}\geq m$ for some $m>0$.
Furthermore, let $P_{1}=P_{1}^{*}\in\R^{d\times d}$ invertible. For
$t\in\R$ we define the fundamental matrix $\Phi_{t}\in C([-\tfrac{1}{2},\tfrac{1}{2}];\mathbb{C}^{d\times d})$
associated to
\[
u'(x)=P_{1}^{-1}(\i t\mathcal{H}(x)^{-1}-P_{0})u(x)\in\mathbb{C}^{d}\quad(x\in(a,b))
\]
subject to $\Phi_{t}(a)=1_{d}$. In this section, we focus on the
condition
\begin{equation}
\sup_{t\in\mathbb{R}}\|\Phi_{t}\|_{\infty}<\infty.\tag{{B}}\label{eq:BPhi}
\end{equation}
Whilst we do not yet know of any counterexamples, we managed to provide
sufficient conditions on $P_{1}$ and $\mathcal{H}$ warranting \prettyref{eq:BPhi}.
These conditions either require some compatibility properties for
$P_{1}$ and $\mathcal{H}$ or regularity properties for $\mathcal{H}$.
In any case, these conditions are somewhat independent of $P_{0}$
as the next result confirms. For this, we use the short-hand $\Phi_{t,P_{0}}$
to denote the above fundamental solution for some fixed $P_{0}$. 
\begin{prop}
\label{prop:wlogP00}In the setting of this section we have
\[
\sup_{t\in\mathbb{R}}\|\Phi_{t,P_{0}}\|_{\infty}<\infty\iff\sup_{t\in\mathbb{R}}\|\Phi_{t,0}\|_{\infty}<\infty.
\]
\end{prop}

We recall an estimate of general nature.
\begin{lem}
\label{lem:wlogP00}Let $\mathcal{O}\in L_{\infty}(a,b;\R^{d\times d})$
and $\Psi\in C([a,b];\R^{d\times d})$ be the fundamental solution
of 
\[
u'(x)=\mathcal{O}(x)u(x)
\]
with $\Psi(a)=1_{d}$. If $f\in L_{1}(a,b;\R^{d})$, then any continuous
solution, $u$, of 
\[
u'(x)=\mathcal{O}(x)u(x)+f(x)
\]
satisfies
\[
\|u(x)\|\leq\|\Psi\|_{\infty}\|u(a)\|+\|\Psi\|_{\infty}\|\Psi(\cdot)^{-1}\|_{\infty}\int_{a}^{x}|f(s)|\d s.
\]
\end{lem}

\begin{proof}
We employ the variations of constants formula
\[
u(x)=\Psi(x)u(a)+\int_{a}^{x}\Psi(x)\Psi(s)^{-1}f(s)\d s,
\]
which can be readily verified. We, thus, estimate
\[
\|u(x)\|\leq\|\Psi\|_{\infty}\|u(a)\|+\|\Psi\|_{\infty}\|\Psi(\cdot)^{-1}\|_{\infty}\int_{a}^{x}|f(s)|\d s.\tag*{{\qedhere}}
\]
\end{proof}
Next we address the fact that $P_{0}$ can, in fact, be assumed to
be $0$.
\begin{proof}[Proof of \prettyref{prop:wlogP00}]
 Let $P_{0}=-P_{0}^{*}$ and assume that $\sup_{t\in\mathbb{R}}\|\Phi_{t,P_{0}}\|_{\infty}<\infty$.
Let $u_{0}\in\R^{d}$ be a unit vector and consider the differential
equation
\begin{equation}
u'(x)=P_{1}^{-1}\i t\mathcal{H}(x)^{-1}u(x),\quad u(a)=u_{0}.\label{eq:P00}
\end{equation}
Denote by $u_{t}$ its solution. Next, let $v_{t}$ be the unique
solution of
\begin{equation}
u'(x)=P_{1}^{-1}(\i t\mathcal{H}(x)^{-1}-P_{0})u(x),\quad u(a)=u_{0}.\label{eq:Pn0}
\end{equation}
Then
\[
u_{t}'(x)-v_{t}'(x)=P_{1}^{-1}(\i t\mathcal{H}(x)^{-1}-P_{0})(u_{t}(x)-v_{t}(x))+P_{1}^{-1}P_{0}u_{t}(x).
\]
By Lemma \ref{lem:wlogP00} and Lemma \ref{lem:inverse_fundamental_matrix},
we infer 
\[
\|u_{t}(x)-v_{t}(x)\|\leq\|\Phi_{t,P_{0}}\|_{\infty}^{2}\|P_{1}\|\|P_{1}^{-1}\|^{2}\|P_{0}\|\int_{a}^{x}\|u_{t}(s)\|\d s.
\]
Hence, using the assumption, we obtain
\begin{align*}
\|u_{t}(x)\| & \leq\|v_{t}(x)\|+\|u_{t}(x)-v_{t}(x)\|\\
 & \leq\|\Phi_{t,P_{0}}\|_{\infty}+\|\Phi_{t,P_{0}}\|_{\infty}^{2}\|P_{1}\|\|P_{1}^{-1}\|^{2}\|P_{0}\|\int_{a}^{x}\|u_{t}(s)\|\d s.
\end{align*}
Gronwall's lemma thus confirms that
\[
\|u_{t}(x)\|\leq\|\Phi_{t,P_{0}}\|_{\infty}\exp\left((b-a)\|\Phi_{t,P_{0}}\|_{\infty}^{2}\|P_{1}\|\|P_{1}^{-1}\|^{2}\|P_{0}\|\right).
\]
Computing the supremum over $t\in\R$ yields the assertion.

Next, let us assume that $\sup_{t\in\mathbb{R}}\|\Phi_{t,0}\|_{\infty}<\infty.$
Similarly, as before, let $u_{0}\in\R^{d}$ be a unit vector and let
$u_{t}$ and $v_{t}$ be the respective solutions of \eqref{eq:P00}
and \eqref{eq:Pn0}. Considering
\[
u_{t}'(x)-v_{t}'(x)=P_{1}^{-1}\i t\mathcal{H}(x)^{-1}(u_{t}(x)-v_{t}(x))+P_{1}^{-1}P_{0}v_{t}(x)
\]
 and estimating as before, we eventually get the assertion as above.
\end{proof}
We may now turn to the structural assumption connecting the positive
and negative spectral subspaces of $P_{1}$ and the mapping properties
of $\mathcal{H}$.

\subsection{A compatibility condition of $\mathcal{H}$ and $P_{1}$}

We start off with a condition irrespective of any regularity of $\mathcal{H}$.
We recall from \prettyref{sec:Preliminaries}
\begin{align*}
E_{+} & =\Span\{x\in\R^{d}\,;\,\exists\lambda>0:P_{1}x=\lambda x\},\\
E_{-} & =\Span\{x\in\R^{d}\,;\,\exists\lambda<0:P_{1}x=\lambda x\}.
\end{align*}
Then $E_{+}\oplus E_{-}=\R^{d}$ in the sense of an orthogonal sum,
since $P_{1}$ is self-adjoint and invertible.

The desired result reads as follows
\begin{thm}
\label{thm:dood2} Assume that, for almost every $x\in(a,b)$,
\[
\mathcal{H}(x)[E_{+}]\subseteq E_{+}.
\]
 Then \prettyref{eq:BPhi} holds.
\end{thm}

\begin{proof}
By \prettyref{prop:wlogP00}, without restriction, we may assume $P_{0}=0$.
We consider the case $E_{+}=\R^{d}$ first. Let $u_{0}\in\R^{d}$
and let $u_{t}$ be the solution of 
\[
u'(x)=P_{1}^{-1}\i t\mathcal{H}(x)^{-1}u(x),\quad u(a)=u_{0}.
\]
Multiplying the equation by $P_{1}^{1/2}$ we obtain
\[
\left(P_{1}^{1/2}u\right)'(x)=\i tP_{1}^{-1/2}\mathcal{H}(x)^{-1}P_{1}^{-1/2}\left(P_{1}^{1/2}u\right)(x).
\]
Hence, the equation satisfied by $u_{t}$ is equivalent to $w=P_{1}^{1/2}u_{t}$
solving
\[
w'(x)=\i tP_{1}^{-1/2}\mathcal{H}(x)^{-1}P_{1}^{-1/2}w(x)\quad w(a)=P_{1}^{1/2}u_{0}.
\]
By self-adjointness of $\mathcal{H}(x)$ and $P_{1}$ it follows that
$tP_{1}^{-1/2}\mathcal{H}(x)^{-1}P_{1}^{-1/2}$ is self-adjoint. Thus,
we deduce
\[
\frac{1}{2}\frac{\d}{\d x}\|w(x)\|^{2}=\Re\langle w(x),\i tP_{1}^{-1/2}\mathcal{H}(x)^{-1}P_{1}^{-1/2}w(x)\rangle=0.
\]
 Thus, $\|w(x)\|\leq\|P_{1}^{1/2}u_{0}\|,$ which proves the assertion
for $E_{+}=\mathbb{R}^{d}$. 

For the general case, it follows that the assumption guarantees that
$\mathcal{H}(x)$ reduces $E_{+}$ and, hence, also $E_{-}$. The
same properties follow for $\mathcal{H}(x)^{-1}$. Hence, the system
is actually block-diagonal, with each block similar to the type considered
in the special case (for the $E_{-}$-block use the previous rationale
multiplying by $\left(P_{1}^{-}\right)^{1/2}$). This shows the assertion.
\end{proof}
\begin{example}
\label{exa:scalarH}Let $\mathcal{H}$ be \emph{scalar}-\emph{valued};
i.e., there exists bounded scalar function $h\colon\left[a,b\right]\to\R$
such that $\inf_{x\in\left[a,b\right]}h(x)>0$ with $\mathcal{H}(x)=h(x)1_{d}$
for almost every $x\in[a,b]$. Then the hypothesis in Theorem \ref{thm:dood2}
is satisfied and hence, \ref{eq:(B)} holds for the corresponding
$(\Phi_{t})_{t}$. 
\end{example}

\begin{proof}[Proof of \prettyref{thm:scalarH}]
 By \prettyref{thm:mrcompressed} we need to look at $\Phi_{t}$
for scalar-valued $\mathcal{H}$. Thus, let $\mathcal{H}=h1_{d}$
for some scalar function $h$. Then differentiation shows that 
\[
\Phi_{t}(x)=\e^{\i t\int_{a}^{x}h(\sigma)^{-1}\d\sigma P_{1}^{-1}}.
\]
As $h$ is scalar, by \prettyref{exa:scalarH}, $t\mapsto\|\Phi_{t}\|_{\infty}$
is bounded. Since $\Phi_{t}(b)=\e^{\i t\int_{a}^{b}h(\sigma)^{-1}\d\sigma P_{1}^{-1}}$
and $\int_{a}^{b}h(\sigma)^{-1}\d\sigma\neq0$, the second condition
\prettyref{thm:scalarH} is equivalent to the second one in \prettyref{thm:mrcompressed}.
This shows the assertion.
\end{proof}
With the results of this section, we can also prove another theorem
from the introduction.
\begin{proof}[Proof of \prettyref{thm:suffcrit-3}]
 The claim follows using \prettyref{exa:scalarH} and \prettyref{thm:suffcrit-1}.
\end{proof}
The next example provides a set-up for which the Hamiltonian density
can be as rough as $L_{\infty}$, but the corresponding port-Hamiltonian
operator still generates an exponentially stable semi-group.
\begin{example}
Let 
\begin{align*}
P_{1}=\left(\begin{array}{cc}
-1 & 0\\
0 & 1
\end{array}\right),\quad W=\left(\begin{array}{cc}
\left(\begin{array}{cc}
\tfrac{1}{2} & 0\\
\tfrac{1}{2} & -1
\end{array}\right) & \left(\begin{array}{cc}
-1 & -\tfrac{1}{2}\\
0 & \tfrac{1}{2}
\end{array}\right)\end{array}\right), & \quad P_{0}=0
\end{align*}
and $\mathcal{H}$ be scalar-valued. Then the corresponding port-Hamiltonian
$\mathcal{A}$ generates a contraction semi-group. Furthermore, using
the formula in (ii) in \prettyref{thm:scalarH}, we get
\begin{align*}
\tau_{t} & =\left(\begin{array}{cc}
\left(\begin{array}{cc}
\tfrac{1}{2} & 0\\
\tfrac{1}{2} & -1
\end{array}\right) & \left(\begin{array}{cc}
-1 & -\tfrac{1}{2}\\
0 & \tfrac{1}{2}
\end{array}\right)\end{array}\right)\left(\begin{array}{c}
\left(\begin{array}{cc}
\e^{-\i t} & 0\\
0 & \e^{\i t}
\end{array}\right)\\
\left(\begin{array}{cc}
1 & 0\\
0 & 1
\end{array}\right)
\end{array}\right)\\
 & =\left(\begin{array}{cc}
\tfrac{1}{2} & 0\\
\tfrac{1}{2} & -1
\end{array}\right)\left(\begin{array}{cc}
\e^{-\i t} & 0\\
0 & \e^{\i t}
\end{array}\right)+\left(\begin{array}{cc}
-1 & -\tfrac{1}{2}\\
0 & \tfrac{1}{2}
\end{array}\right)\left(\begin{array}{cc}
1 & 0\\
0 & 1
\end{array}\right)\\
 & =\left(\begin{array}{cc}
\tfrac{1}{2}\e^{-\i t}-1 & -\tfrac{1}{2}\\
\tfrac{1}{2}\e^{-\i t} & \tfrac{1}{2}-\e^{\i t}
\end{array}\right).
\end{align*}
Next, 
\begin{align*}
\det\tau_{t} & =(\tfrac{1}{2}\e^{-\i t}-1)(\tfrac{1}{2}-\e^{\i t})+\tfrac{1}{4}\e^{-\i t}=\tfrac{1}{2}\e^{-\i t}-1+\e^{\i t}.
\end{align*}
This expression is $2\pi$-periodic. It thus, suffices to consider
$t\in[0,2\pi)$. Since $\Im\det\tau_{t}=\frac{1}{2}\sin t=0$ if and
only if $t\in\left\{ 0,\pi\right\} $. For these values, however,
we have $\det\tau_{0}=\frac{1}{2}$ and $\det\tau_{\pi}=-\frac{5}{2}.$
By continuity, $\min_{t\in\R}\left|\det\tau_{t}\right|=\min_{t\in[0,2\pi]}\left|\det\tau_{t}\right|>0.$
Thus, Cramer's rule implies that $\tau_{t}$ is invertible with uniform
bound for the inverse. Hence, the corresponding port-Hamiltonian semi-group
is exponentially stable. 
\end{example}

\subsection{A regularity condition on $\mathcal{H}$ \label{subsec:A-regularity-condition}}

The aim of this subsection is to show that \prettyref{thm:suffcrit-1}
contains all the cases already contained in \prettyref{thm:suffcrit}.
For this recall the setting from the beginning of \prettyref{sec:CondB};
also we briefly define what it means for a function to be of bounded
variation.
\begin{defn}
Let $I\subseteq\mathbb{R}$ be an open interval, $\alpha\in L_{1}(I)$.
Then $\alpha$ is of \emph{bounded variation}, if 
\[
\sup\{|\int_{I}\alpha(x)\phi'(x)\d x|;\phi\in C_{c}^{1}(I),\|\phi\|_{\infty}\leq1\}<\infty.
\]
By the Riesz--Markov representation theorem, there exists a unique
signed Radon measure, $\D_{x}\alpha$, on the Borel sets of $I$ such
that 
\[
-\int_{I}\alpha\phi'=\int_{I}\phi\d\D_{x}\alpha\eqqcolon\langle\D_{x}\alpha,\phi\rangle,
\]
for each $\phi\in C_{c}^{1}(I)$. Moreover 
\[
\|\mathrm{D}_{x}\alpha\|\coloneqq|\mathrm{D}_{x}\alpha|(I)=\sup\{|\int_{I}\alpha(x)\phi'(x)\d x|;\phi\in C_{c}^{1}(I),\|\phi\|_{\infty}\leq1\},
\]
where $|\D_{x}\alpha|$ denotes the total variation of the measure
$\D_{x}\alpha.$ The space 
\[
BV(I)\coloneqq\{\alpha\in L_{1}(I);\alpha\text{ of bounded variation}\}
\]
becomes a Banach space, if endowed with the norm given by
\[
\|\alpha\|_{BV}\coloneqq\|\alpha\|_{L_{1}}+\|\D_{x}\alpha\|.
\]
A matrix-valued function $G\colon(a,b)\to\mathbb{R}^{d\times d}$
is called \emph{of bounded variation,} if all its components are of
bounded variation. In this case we set
\[
\D_{x}G\coloneqq\left(\D_{x}G_{ij}\right)_{i,j\in\{1,\ldots,n\}}
\]
as a matrix-valued measure.
\end{defn}

\begin{thm}
\label{thm:HBV}Assume that $\mathcal{H}$ is of bounded variation.
Then \ref{eq:(B)} is satisfied.
\end{thm}

This results immediately yields a proof for \prettyref{prop:constantH_0}:
\begin{proof}[Proof of \prettyref{prop:constantH_0}]
 Any constant is of bounded variation. Thus, the result follows from
\prettyref{thm:HBV}.
\end{proof}
For the proof of \prettyref{thm:HBV}, some preliminaries are in order.
The material is widely known. We shall, however, summarise and prove
some particular findings needed in the present situation. Note that
the author of \cite{Leoni2009} focuses on right-continuous instead
of left-continuous functions. The arguments, however, are similar
in either cases so we still quote the results without proof.
\begin{thm}[{\cite[Theorem 7.2 and Theorem 5.13]{Leoni2009}}]
 \label{thm:LeftcontBV}Let $I=(a,b)\subseteq\mathbb{R}$ be an interval
and $\alpha\in L_{1}(I)$. Then the following conditions are equivalent:\begin{enumerate}

\item[(i)] $\alpha\in BV(I)$,

\item[(ii)] there exists a left-continuous representative, $\alpha_{\textnormal{lc}}$,
of $\alpha$ such that
\[
\var(\alpha_{\textnormal{lc}})\coloneqq\sup_{a<t_{0}<t_{1}<\cdots<t_{n}<b}\sum_{j=1}^{n}|\alpha(t_{j})-\alpha(t_{j-1})|<\infty.
\]
\end{enumerate}In either case, $\var(\alpha_{\textnormal{lc}})=\|\D_{x}\alpha\|$
and $\alpha_{\textnormal{lc}}$ may be chosen according to
\[
\alpha_{\textnormal{lc}}(x)\coloneqq\alpha(t_{0})+\begin{cases}
\D_{x}\alpha(([t_{0},x)), & x>t_{0},\\
-\D_{x}\alpha([x,t_{0})), & x\leq t_{0},
\end{cases}
\]
for a Lebesgue point $t_{0}\in(a,b)$ of $\alpha$.
\end{thm}

An immediate consequence of the previous theorem is that if the Hamiltonian
energy density $\mathcal{H}$ is of bounded variation, the same is
true for $\mathcal{H}^{-1}$.
\begin{prop}
\label{prop:inverseBV}Let $\mathcal{O}\colon(a,b)\to\mathbb{R}^{d\times d}\in L_{\infty}(a,b;\R^{d\times d})$
be of bounded variation and assume that $\Re\mathcal{O}(x)\geq m$
for some $m>0$ and a.e.~$x\in(a,b).$ Then $x\mapsto\mathcal{O}(x)^{-1}$
is of bounded variation.
\end{prop}

\begin{proof}
Since every matrix entry of $\mathcal{O}$ is in $L_{1}$, we may
choose a common Lebesgue point $t_{0}$ for all matrix entries. By
Theorem \ref{thm:LeftcontBV}, the function given by 
\[
\mathcal{O}_{\textnormal{lc}}(x)\coloneqq\mathcal{O}(t_{0})+\left(\begin{cases}
\D_{x}\mathcal{O}_{i,j}([t_{0},x)), & x>t_{0},\\
-\D_{x}\mathcal{O}_{i.j}([x,t_{0})), & x\leq t_{0}
\end{cases}\right)_{i,j}
\]
defines a left-continuous representative of $\mathcal{O}$. By composition,
$x\mapsto\mathcal{O}_{\textnormal{lc}}(x)^{-1}$ is, too, left-continuous
and evidently it is a representative of $x\mapsto\mathcal{O}(x)^{-1}$.
Since $\|\mathcal{O}(x)^{-1}\|\leq1/m$ by $\Re\mathcal{O}(x)\geq m$
and the boundedness of $(a,b)$ it follows that $x\mapsto\mathcal{O}(x)^{-1}\in L_{1}(a,b;\R^{d})$.
Next, let $a<x_{0}<\cdots<x_{n}<b$ and compute
\begin{align*}
\sum_{j=1}^{n}\|\mathcal{O}_{\textnormal{lc}}(x_{j})^{-1}-\mathcal{O}_{\textnormal{lc}}(x_{j-1})^{-1}\| & \leq\sum_{j=1}^{n}\|\mathcal{O}_{\textnormal{lc}}(x_{j})^{-1}\left(\mathcal{O}_{\textnormal{lc}}(x_{j-1})-\mathcal{O}_{\textnormal{lc}}(x_{j})\right)\mathcal{O}_{\textnormal{lc}}(x_{j-1})^{-1}\|\\
 & \leq\sum_{j=1}^{n}\|\mathcal{O}_{\textnormal{lc}}(x_{j})^{-1}\|\|\left(\mathcal{O}_{\textnormal{lc}}(x_{j-1})-\mathcal{O}_{\textnormal{lc}}(x_{j})\right)\|\|\mathcal{O}_{\textnormal{lc}}(x_{j-1})^{-1}\|\\
 & \leq\frac{1}{m^{2}}\sum_{j=1}^{n}\|\left(\mathcal{O}_{\textnormal{lc}}(x_{j-1})-\mathcal{O}_{\textnormal{lc}}(x_{j})\right)\|.
\end{align*}
Thus, if $\kappa>0$ is such that $\|A\|\leq\kappa\sum_{i,j}|A_{i,j}|$,
for every $k,l\in\{1,\ldots,d\}$, 
\[
\var\left(\mathcal{O}_{\textnormal{lc}}(\cdot)^{-1}\right)_{k,l}\leq\frac{\kappa}{m^{2}}\sum_{i,j}\var\mathcal{O}_{\textnormal{lc}}(\cdot)_{i,j}<\infty.
\]
Hence, by Theorem \ref{thm:LeftcontBV}, the assertion follows.
\end{proof}
\begin{thm}
\label{thm:productruleBVH1}Let $I\subseteq\mathbb{R}$ be an open
and bounded interval, $\alpha\in BV(I)\cap L_{\infty}(I).$ If $u\in H^{1}(I)$,
then $\alpha u\in BV(I)\cap L_{\infty}(I)$ and
\[
\D_{x}(\alpha u)=u\D_{x}\alpha+u'\alpha\d\lambda,
\]
where $\d\lambda$ denotes the Lebesgue measure and $u\D_{x}\alpha$
is the measure $\D_{x}\alpha$ with density $u$. Moreover, we have
\[
\|\D_{x}(\alpha u)\|\leq\|\D_{x}\alpha\|\|u\|_{H^{1}}+\|\alpha\|_{\infty}\|u\|_{H^{1}}\sqrt{b-a}.
\]
\end{thm}

\begin{proof}
Let $\phi\in C_{c}^{1}(I)$. Assume that $u$ is continuously differentiable.
Then we compute 
\begin{align*}
-\int_{I}\alpha u\phi'd\lambda & =-\int_{I}\alpha((u\phi)'-u'\phi)\d\lambda\\
 & =\langle\D_{x}\alpha,u\phi\rangle+\int_{I}u'\alpha\phi\d\lambda\\
 & =\langle u\D_{x}\alpha+u'\alpha\d\lambda,\phi\rangle
\end{align*}
We estimate
\begin{align*}
|\int_{I}\alpha u\phi'\d\lambda| & \leq\|\D_{x}\alpha\|\|u\phi\|_{\infty}+\|\alpha\|_{\infty}\|u'\phi\|_{L_{1}}\\
 & \leq\|\D_{x}\alpha\|\|u\|_{\infty}\|\phi\|_{\infty}+\|\alpha\|_{\infty}\|u'\|_{L_{2}}\|\phi\|_{L_{2}}\\
 & \leq\|\D_{x}\alpha\|\|u\|_{H^{1}}\|\phi\|_{\infty}+\|\alpha\|_{\infty}\|u\|_{H^{1}}\sqrt{b-a}\|\phi\|_{\infty}\\
 & =\left(\|\D_{x}\alpha\|\|u\|_{H^{1}}+\|\alpha\|_{\infty}\|u\|_{H^{1}}\sqrt{b-a}\right)\|\phi\|_{\infty}.
\end{align*}
Hence, $\alpha u\in BV(I)$ and 
\[
\|\D_{x}(\alpha u)\|\leq\|\D_{x}\alpha\|\|u\|_{H^{1}}+\|\alpha\|_{\infty}\|u\|_{H^{1}}\sqrt{b-a}.
\]
Moreover, we estimate
\[
\|\alpha u\|_{L_{1}}\leq\|u\|_{\infty}\|\alpha\|_{L_{1}}\leq c\|u\|_{H^{1}}\|\alpha\|_{L_{1}},
\]
by the Sobolev embedding theorem. In particular, let $(u_{n})_{n}$
be a sequence of continuously differentiable functions converging
to $u$ in $H^{1}$. Then, by the estimates above, $\alpha u\in BV(I)$.
Moreover, using the product formula from the beginning of the proof
for $u_{n}$ and letting $n\to\infty$, we infer
\[
\D_{x}(\alpha u)=u\D_{x}\alpha+u'\alpha\d\lambda.\tag*{{\qedhere}}
\]
\end{proof}
Next, we recall Gronwall's inequality for locally finite measures.
\begin{defn}
Let $I\subseteq\R$ be an interval. We call a Borel measure $\mu$
on $I$ \emph{locally finite}, if for all compact $K\subseteq I$,
$\mu(K)<\infty$. 
\end{defn}

\begin{thm}[{\cite[Lemma A.1]{Seifert2013}}]
\label{thm:Groenwald}Let $I\subseteq\R$ be an interval. Let $u\colon I\to\R$
and $\alpha\colon I\to[0,\infty)$ measurable. Assume that for a locally
finite Borel measure $\mu$ on $I$, we have $u$ is locally integrable
and there exists $a\in I$ such that for all $t>a$
\[
u(t)\leq\alpha(t)+\int_{[a,t)}|u(s)|\d\mu(s).
\]
Then 
\[
u(t)\leq\alpha(t)+\int_{[a,t)}\alpha(s)\exp(\mu((s,t)))\d\mu(s)\quad(t>a).
\]
\end{thm}

Now, we are in the position to show the main result of this subsection.
\begin{proof}[Proof of \prettyref{thm:HBV}]
 Let $v\in H^{1}(a,b)^{d}$ and $t\in\R$ satisfy 
\[
v'(x)=\i tP_{1}^{-1}\mathcal{H}(x)^{-1}v(x),\quad\left(x\in(a,b)\right).
\]
By Proposition \prettyref{prop:inverseBV}, $\mathcal{O}\colon x\mapsto\mathcal{H}(x)^{-1}$
is of bounded variation. Referring to \prettyref{thm:LeftcontBV},
without loss of generality, we may assume that $\mathcal{O}$ is left-continuous.
Next, note that 
\begin{align*}
\langle v(x),\mathcal{O}(x)v(x)\rangle_{\C^{d}} & =\sum_{j=1}^{d}\sum_{k=1}^{d}v_{j}(x)^{*}\mathcal{O}_{jk}(x)v_{k}(x)\\
 & =\langle\mathcal{O}(x),\left(v_{j}(x)^{*}\right)_{j}v(x)^{\top}\rangle_{\C^{d\times d}}
\end{align*}
for all $x\in(a,b)$. Thus, by \prettyref{thm:productruleBVH1}, we
obtain
\begin{align*}
\D_{x}\langle v,\mathcal{O}v\rangle_{\C^{d}} & =\langle\mathcal{O},\partial_{x}(\left(v_{j}(\cdot)^{*}\right)_{j}v^{\top})\rangle\d\lambda+\langle\textrm{D}_{x}\mathcal{O},vv^{\top}\rangle\\
 & =\left(\langle v',\mathcal{O}v\rangle+\langle v,\mathcal{O}v'\rangle\right)\d\lambda+\langle\textrm{D}_{x}\mathcal{O},vv^{\top}\rangle\\
 & =\left(\langle\i tP_{1}^{-1}\mathcal{O}v,\mathcal{O}v\rangle+\langle\mathcal{O}v,\i tP_{1}^{-1}\mathcal{O}v\rangle\right)\d\lambda+\langle\textrm{D}_{x}\mathcal{O},vv^{\top}\rangle\\
 & =\langle\D_{x}\mathcal{O},vv^{\top}\rangle.
\end{align*}
Since $v$ is continuous by the Sobolev embedding theorem, $\langle v,\mathcal{O}v\rangle_{\C^{d}}$
is also left-continuous. For $s,t\in(a,b)$ with $s<t$ we therefore
obtain by \prettyref{thm:LeftcontBV}
\begin{align*}
\langle v(t),\mathcal{O}(t)v(t)\rangle_{\C^{d}} & =\langle v(s),\mathcal{O}(s)v(s)\rangle+\D_{x}\langle v,\mathcal{O}v\rangle_{\C^{d}}([s,t))\\
 & =\langle v(s),\mathcal{O}(s)v(s)\rangle+\sum_{j=1}^{d}\sum_{k=1}^{d}\int_{[s,t)}v_{j}(\sigma)^{*}v_{k}(\sigma)\d\D_{x}\mathcal{O}_{jk}(\sigma)
\end{align*}
Hence, using our general assumption on boundedness of $\mathcal{H}$
and $\mathcal{H}(x)\geq m$, for some $c>0$ we estimate 
\begin{align*}
c\|v(t)\|^{2} & \leq\langle v(t),\mathcal{O}(t)v(t)\rangle\\
 & \leq\|v(s)\|^{2}\frac{1}{m}+\frac{1}{2}\int_{(s,t]}\|v(\sigma)\|^{2}\d\mu(\sigma),
\end{align*}
where $\mu\coloneqq\sum_{j,k=1}^{d}|\D_{x}\mathcal{O}_{jk}|$. Note
that $\mu$ is a finite measure on $(a,b)$. Using \prettyref{thm:Groenwald},
we get
\begin{align*}
\|v(t)\|^{2} & \leq\tfrac{1}{cm}\|v(s)\|^{2}\left(1+\frac{1}{2c}\int_{[s,t)}\e^{\frac{1}{2c}\mu\left((\sigma,t)\right)}\d\mu(\sigma)\right)\\
 & \leq\tfrac{1}{cm}\|v(s)\|^{2}\left(1+\tfrac{1}{2c}\e^{\frac{1}{2c}\mu((a,b))}\mu((a,b))\right),
\end{align*}
Letting now $s\to a$ in the last inequality, we derive 
\[
\|v(t)\|\leq C\|v(a)\|^{2}\quad(t\in[a,b])
\]
for some constant $C>0$, proving the desired result.
\end{proof}
We finally obtain a proof of \prettyref{thm:suffcrit}.
\begin{proof}[Proof of \prettyref{thm:suffcrit}]
 The statement is an immediate consequence of \prettyref{thm:suffcrit-1}
and \prettyref{thm:HBV}.
\end{proof}

\section{Conclusion\label{sec:Conclusion}}

We presented a new characterisation of exponential stability for port-Hamiltonian
systems. The characterisation works only if a certain family of fundamental
solutions of a non-autonomous ODE-system is uniformly bounded. Whether
or not this boundedness is needed lies beyond the scope of this article
and can be considered an \textbf{open problem }for the time-being.
We emphasise that as the strategy above uses existence and boundedness
of the resolvent of the generator on the imaginary axis it might be
possible, invoking results such as \cite{Arendt1988,Batty_Duyckaerts_2008},
to show explicit algebraic decay for certain set-ups. Whether at all
these set-ups exists and how they maybe characterised will be addressed
in future work.

\section*{Acknowledgements}

We thank the anonymous referee of a former version (and submission)
of this manuscript for spotting a mistake, which eventually led to
the present major revision containing corrected statements as well
as leaner arguments and proofs. We also thank the anonymous referee
of the actual submission, in particular for providing the nice and
short proof of \prettyref{prop:fundamental_matrix} (a).

\end{document}